\documentclass[preprint,12pt]{elsarticle}

\usepackage{amssymb}
\usepackage{graphicx,amsthm,enumitem}
\usepackage{
	nameref,
	hyperref,
	cleveref,
}

\newtheorem{theorem}{Theorem}[section]
\newtheorem{lemma}[theorem]{Lemma}
\newtheorem{prop}[theorem]{Proposition}
\newtheorem{corollary}[theorem]{Corollary}

\newcommand{\abs}[1]{\lvert#1\rvert}

\newcommand{\abss}[1]{\|#1\|}
\newcommand{\comment}[1]{}

\usepackage{mathtools}
\DeclareMathOperator{\supp}{supp}
\allowdisplaybreaks
\makeatletter
\DeclareRobustCommand\widecheck[1]{{\mathpalette\@widecheck{#1}}}
\def\@widecheck#1#2{%
	\setbox\z@\hbox{\m@th$#1#2$}%
	\setbox\tw@\hbox{\m@th$#1%
		\widehat{%
			\vrule\@width\z@\@height\ht\z@
			\vrule\@height\z@\@width\wd\z@}$}%
	\dp\tw@-\ht\z@
	\@tempdima\ht\z@ \advance\@tempdima2\ht\tw@ \divide\@tempdima\thr@@
	\setbox\tw@\hbox{%
		\raise\@tempdima\hbox{\scalebox{1}[-1]{\lower\@tempdima\box
				\tw@}}}%
	{\ooalign{\box\tw@ \cr \box\z@}}}
\makeatother

\newcommand\numberthis{\addtocounter{equation}{1}\tag{\theequation}}

\journal{Journal of Differential Equations}

\begin{document}

\begin{frontmatter}



    \title{Small Data Well-Posedness for Derivative Nonlinear Schr{\"o}dinger Equations\tnoteref{label1}}
    \tnotetext[label1]{\copyright 2018. This manuscript version is made available under the CC-BY-NC-ND 4.0 license \url{http://creativecommons.org/licenses/by-nc-nd/4.0/}}

\author{Donlapark Pornnopparath}

\address{Department of Mathematics, University of California, San Diego, La Jolla, CA 92093-0112, USA}

\ead{donlapark@ucsd.edu}
\begin{abstract}
We study the local and global solutions of the generalized derivative nonlinear Schr{\"o}dinger equation $i\partial_t u+\Delta u = P(u,\overline{u},\partial_x u,\partial_x \overline{u})$, where each monomial in $P$ is of degree $3$ or higher, in low-regularity Sobolev spaces without using a gauge transformation. Instead, we use a solution decomposition technique introduced in \cite{BeIoKeTa} during the perturbative argument to deal with the loss on derivative in nonlinearity. It turns out that when each term in $P$ contains only one derivative, the equation is locally well-posed in $H^{\frac{1}{2}}$, otherwise we have a local well-posedness in $H^{\frac{3}{2}}$. If each monomial in $P$ is of degree $5$ or higher, the solution can be extended globally. By restricting to equations to the form $i\partial_t u+\Delta u = \partial_x P(u,\overline{u})$ with the quintic nonlinearity, we were able to obtain the global well-posedness in the critical Sobolev space. 
\end{abstract}


\end{frontmatter}


\section{Introduction}\label{sec1}
\noindent
In this paper, we study the well-posedness of the Cauchy problem for the generalized derivative nonlinear Schr{\"o}dinger equation (gDNLS) on $\mathbb{R}$.
\begin{equation}\label{dnls2}
\begin{cases}
i\partial_t u+\Delta u = P(u,\overline{u},\partial_x u,\partial_x\overline{ u}) \\
u(x,0)=u_0\in H^s(\mathbb{R}), s\geq s_0.
\end{cases}
\end{equation}
Here, $u$ is a complex-valued function and $P:\mathbb{C}^4 \to \mathbb{C}$ is a polynomial of the form 
\begin{equation}\label{poly}
P(z)=P(z_1,z_2,z_3,z_4)= \sum_{d\leq \abs{\alpha}\leq l }C_{\alpha}z^{\alpha},
\end{equation} 
and $l\geq d\geq 3$. There are several results regarding the well-posedness of this equation. In \cite{KenigPonceVega}, Kenig, Ponce and Vega proved that the equation \eqref{dnls2} is locally well-posed for a small initial data in $H^{\frac{7}{2}}(\mathbb{R})$. There has been some interest in the special case where $P=i\lambda\abs{u}^ku_x$: 
\begin{equation*}
\begin{cases}
    i\partial_t u+\Delta u = i\lambda\abs{u}^ku_x \\
    u(x,0) =u_0\in H^s(\mathbb{R}), s\geq s_0.
\end{cases}
\end{equation*}
with $k\in\mathbb{R}$. Hao (\cite{Hao}) proved that this equation is locally well-posed in $H^{\frac{1}{2}}(\mathbb{R})$  for $k\geq 5$, and Ambrose-Simpson (\cite{AmSim}) proved the result in $H^1(\mathbb{R})$ for $k\geq 2$. Recent studies show that these results can be improved. See Santos (\cite{Santos}) for the local-wellposedness in $H^{\frac{1}{2}}$ when $k\geq 2$ and Hayashi-Ozawa (\cite{HaOz02}) for the local well-posedness in $H^2$ when $k\geq 1$ and the global well-posedness in $H^1$ when $k\geq 2$. \\
\\
Several studies showed that we have better results if $P$ only consists of $\overline{u}$ and $\partial_x\overline{u}$ due to the following heuristic: if $u$ solves the linear Schr{\"o}dinger equation, then the space-time Fourier transform of $\overline{u}$ is supported away from the parabola $\{(\xi,\tau)|\tau+\xi^2=0\}$, leading to strong dispersive estimates. Gr{\"u}nrock (\cite{Grun}) showed that for $P=\partial_x(\overline{u}^d)$ or $P=(\partial_x\overline{u})^d$ where $d\geq 3$, the equation \eqref{dnls2} is locally well-posed for any $s>\frac{1}{2}-\frac{1}{d-1}$ in the former case and $s>\frac{3}{2}-\frac{1}{d-1}$ in the latter. Later, Hirayama (\cite{Hirayama}) extended Gr{\"u}nrock's results for $P=\partial_x(\overline{u}^d)$ to the global well-posedness for $s\geq \frac{1}{2}-\frac{1}{d-1}$.\\
\\
There are also various results for higher dimension analogues of \eqref{dnls2}
\begin{equation}\label{dnls3}
\begin{cases}
i\partial_t u+\Delta u = P(u,\overline{u},\nabla  u,\nabla\overline{ u}) \\
u(x,0)=u_0(x), \ \ x\in \mathbb{R}^n.
\end{cases}
\end{equation}
The most general results in $\mathbb{R}^n$ for $n\geq 2$ is due to Kenig, Ponce and Vega in \cite{KenigPonceVega}. For a more specific case, we refer to \cite{be1} and \cite{be2} where Bejenaru obtained a local well-posedness result for $n=2$ and $P(z)$ is quadratic with low regularity initial data. For results in Besov spaces, see \cite{wang} for the global well-posedness in $\dot{B}^{s_n}_{1,2}(\mathbb{R}^n)$ where $n\geq 2$ and $s_n=\frac{n}{2}-\frac{1}{d-1}$ which is the critical exponent.\\
\\
\noindent
For another type of derivative nonlinearities, we refer to Chihara (\cite{chi}) for nonlinearities of the form $f(u,\partial u)$, where $f:\mathbb{R}^2\times \mathbb{R}^{2n}\to \mathbb{R}$ (identifying $\mathbb{C}$ with $\mathbb{R}^2$) is a smooth function such that $f(u,v)=O(\abs{u}^2+\abs{v}^2)$ or $f(u,v)=O(\abs{u}^3+\abs{v}^3)$ near $(u,v)=0$. It turns out that the corresponding Cauchy problems are locally well-posed in $H^{\lfloor n/2\rfloor +4}$ for any $n\geq 1$.\\
\\
Our first result is the local well-posedness of \eqref{dnls2} in Sobolev spaces when the nonlinearity contains an arbitrary number of derivatives. 
\begin{theorem}\label{thmm1}
	In the equation \eqref{dnls2}, let $s$ be any number such that
	\begin{enumerate}[label=(\Alph*)]
		\item $s\geq \frac{1}{2}$ if each term in $P(u,\overline{u},\partial_x u,\partial_x \overline{u}) $ has only one derivative,
		\item $s\geq \frac{3}{2}$ if a term in $P(u,\overline{u},\partial_x u,\partial_x \overline{u}) $ has more than one derivative.
	\end{enumerate}
    Then there exist a Banach space $X^s$ and a constant $C=C(s,d)$ with the following properties: For any $u_0\in H^s(\mathbb{R})$ such that $\abss{u_0}_{H^s}< C$, the equation \eqref{dnls2} has a unique solution:
    \[u\in X:=\{u\in C_t^0H_x^s([-1,1]\times \mathbb{R}) \cap X^s : \abss{u}_{X^s}\leq 2C\}.\] 
    Furthermore, the map $u_0 \mapsto u$ is Lipschitz continuous from $B_C := \{u_0\in H^s : \abss{u_0}_{H^s}\leq C\}$ to $X$. 
\end{theorem}
\noindent
\emph{Remark:} The definition of $X^s$ will be made precise in \Cref{sec4} below. \\
\\
This shows that, without any restriction to the number of derivatives, we are able to improve Kenig et al.'s result (\cite{KenigPonceVega}) from $H^{\frac{7}{2}}$ to $H^{\frac{3}{2}}$. By restricting to only one derivative per term in the nonlinearity, we can improve further to $H^{\frac{1}{2}}$. Moreover, part $(A)$ of \Cref{thmm1} extends Hao and Santos's local well-posedness result in $H^{\frac{1}{2}}$ to more general class of nonlinearities. It turns out that the global well-posedness results can be achieved if the nonlinearity is quintic or higher and the endpoint cases are excluded.
\begin{theorem}\label{gwp1}
	Suppose that $d\geq 5$ in \eqref{poly}. Let $s$ be any number such that
	\begin{enumerate}[label=(\Alph*)]
		\item $s> \frac{1}{2}$ if each term in $P(u,\overline{u},\partial_x u,\partial_x \overline{u}) $ has only one derivative,
		\item $s> \frac{3}{2}$ if a term in $P(u,\overline{u},\partial_x u,\partial_x \overline{u}) $ has more than one derivative.
	\end{enumerate}
	Then the equation \eqref{dnls2} is globally well-posed in the following sense: \\
	\\
	There exist a Banach space $X^s$ and a constant $C=C(s,d)$ with the following properties: For any $u_0\in H^s(\mathbb{R})$ such that $\abss{u_0}_{H^s}< C$ and any time interval $I$ containing $0$, the equation \eqref{dnls2} has a unique solution:
    \[u\in X:=\{u\in C_t^0H_x^s(I\times \mathbb{R}) \cap X^s : \abss{u}_{X^s}\leq 2C\}.\] 
    Furthermore, the map $u_0 \mapsto u$ is Lipschitz continuous from $B_C := \{u_0\in H^s : \abss{u_0}_{H^s}\leq C\}$ to $X$. 
\end{theorem}
\noindent
\emph{Remark:} The definition of $X^s$ will be made precise in \Cref{proofgwp} below. \\
\\
Notice that when each term in $P(u,\overline{u},\partial_x u,\partial_x \overline{u}) $ has only one derivative, \eqref{dnls2} is invariance under the scaling $u(x,t)\mapsto u_{\lambda}(x,t):= \lambda^{\frac{1}{d-1}}u(\lambda x,\lambda^2 t)$. Thus, the critical space is $H^{s_0}$ where $s_0=\frac{1}{2}-\frac{1}{d-1}$ in the sense that $\abss{u}_{H^{s_0}}=\abss{u_{\lambda}}_{H^{s_0}}$. If we follow the heuristic that a dispersive equation is expected to be locally well-posed in any subcritical Sobolev space $H^s$ i.e. $s>s_0$, then the result in part $(A)$ of \Cref{gwp1}, which requires $s>\frac{1}{2}$, is not optimal in this sense. It turns out that the global well-posedness at critical Sobolev spaces can be achieved if we assume a specific type of the gDNLS equation
\begin{equation}\label{dnls1}
\begin{cases}
i\partial_t u+\Delta u = \partial_x P(u,\overline{u}) \\
u(x,0)=u_0\in H^s(\mathbb{R}), s\geq s_0.
\end{cases}
\end{equation}
where $P:\mathbb{C}^2 \to \mathbb{C}$ is a polynomial of the form
\begin{equation}\label{1234}
P(z)=P(z_1,z_2)= \sum_{d\leq \abs{\alpha}\leq l }C_{\alpha}z^{\alpha},
\end{equation}
and $l\geq d\geq 5$. \\
\\
\noindent
The following theorem shows that for $d\geq 5$ we have the global well-posedness at the scaling critical Sobolev space. 
\begin{theorem}\label{thmm}
    Suppose that $d\geq 5$ in \eqref{1234}. Let $s_0=\frac{1}{2}-\frac{1}{d-1}$. For any $s\geq s_0$, the equation \eqref{dnls1} is globally well-posed in $H^s(\mathbb{R})$ in the following sense: \\
	\\
	There exist a Banach space $X^s$ and a constant $C=C(s,d)$ with the following properties: For any $u_0\in H^s(\mathbb{R})$ such that $\abss{u_0}_{H^s}< C$ and any time interval $I$ containing $0$, the equation \eqref{dnls1} has a unique solution:
    \[u\in X:=\{u\in C_t^0H_x^s(I\times \mathbb{R}) \cap X^s : \abss{u}_{X^s}\leq 2C\}.\] 
    Furthermore, the map $u_0 \mapsto u$ is Lipschitz continuous from $B_C := \{u_0\in H^s : \abss{u_0}_{H^s}\leq C\}$ to $X$. \\
	\\
	In the case of $s=s_0$, the statement above holds true if we replace $H^s$ by $\dot{H}^{s_0}$.
\end{theorem}
\noindent
\emph{Remark:} The definition of $X^s$ will be made precise in \Cref{sec6} in the case of $d\geq 6$ and \Cref{sec7} in the case of $d=5$ below. \\
\\
This extends Gr{\"u}nrock and Hirayama's results to more general class of nonlinearities. The main ideas behind the proof of \Cref{thmm1} and \Cref{thmm} consist of the Duhamel reformulation of the problem, followed by the contraction argument, using the local smoothing estimate \eqref{stri2} and the maximal function estimate \eqref{li1} to deal with the loss of derivative in nonlinearity.  We also use a decomposition \eqref{dec2} of the nonlinear Duhamel term, first introduced in \cite{BeIoKeTa}, to deal with the truncated time integration. We then finish with the usual perturbative analysis to obtain the well-posedness results. The proof for \Cref{thmm} in the case $d=5$ is rather delicate and needs some modulation-frequency argument, motivated by Tao's paper on the quartic generalised KdV equation (\cite{MR2286393}), which is sensitive to the conjugates in the nonlinearity. Therefore, the proof of global well-posedness in this case will be treated separately in section \ref{sec7}.\\
\\
\noindent
One motivation of this paper came from the following specific case of \eqref{dnls1}, which has been intensively studied in the past:
\begin{equation}\label{dnls}
\begin{cases}
i\partial_t u+\Delta u = i\partial_x(\abs{u}^2u) \\
u(x,0)=u_0\in H^s(\mathbb{R}), s\geq \frac{1}{2}.
\end{cases}
\end{equation}
We name this equation \emph{DNLS}. It arises from studies of small-amplitude Alf{\'v}en waves propagating parallel to a magnetic field \cite{C} and large-amplitude magnetohydrodynamic waves in plasmas \cite{Ruderman1}. There is also recent discovery of rogue waves as solutions for the Darboux transformation of the DNLS (See ~\cite{XuHeWang}). Although one expects the local well-posedness for $s\geq 0$, Biagioni and Linares (\cite{BiLi01}) have showed that (\ref{dnls}) is ill-posed for $s<\frac{1}{2}$ in the sense that the solution mapping $u_0\mapsto u$ fails to be uniformly continuous. This means that our result from \Cref{thmm} when $d=3$, which is a local well-posedness in $H^{\frac{1}{2}}$, is sharp in this sense.  \\
\\
We mention here a few of many results regarding this equation. The global well-posedness in the energy space $H^1(\mathbb{R})$ was proved by Hayashi and Ozawa in \cite{HaOz1}. For data below the energy space, Takaoka has shown in \cite{Takaoka99} that DNLS is locally well-posed for $s\geq \frac{1}{2}$ using \eqref{gt} with $k=-1$. In \cite{Iteam02}, Colliander, Keel, Staffilani, Takaoka and Tao used the ``I-method" to show the global well-posedness of DNLS for $s>\frac{1}{2}$, assuming the smallness condition $\abs{u_0}_{L^2}<\sqrt{2\pi}$. Later, Miao, Wu and Xu have proved the global well-posedness result for the endpoint case $s=\frac{1}{2}$ using the third generation I-method and same smallness condition in \cite{Miao10}. Lastly, Wu (\cite{Wu13} and \cite{Wu14}) has shown that in the energy-critical case $s=1$, the smallness threshold is improved to $\abss{u_0}_L^2< 2\sqrt{\pi}$. \\
We are now shifting focus toward some qualitative aspects of the solutions. Kaup and Newell has shown that the equation in completely integrable, which implies infinitely many conservation laws.
Moreover, the inverse scattering method can be applied to obtain soliton solutions which are unstable in a sense that a small perturbation could cause the soliton to disperse (See \cite{KaNe01}). Recently, Liu, Perry and Sulem used this method to prove the global well-posedness result in $H^{2,2}(\mathbb{R})$ (see \cite{LiuSimpSulem1}). A study following Wu's above result (\cite{CherSimpSul1}) shows an existence of two kinds of solitons: bright solitons with mass $\sqrt{2\pi}$, and lump soliton with mass $2\sqrt{\pi}$. He showed in \cite{Wu13} that there is no blow-up near the $\sqrt{2\pi}$ threshold. On the other hand, the study of Cher, Simpson and Sulem (\cite{CherSimpSul1}) has shown some numerical evidence of a blow-up profile that closely resembles the lump soliton.  \\
\\
The main difficulty in studying DNLS is the spatial derivative in nonlinearity. Due to this, all of well-posedness results for DNLS so far involve the \emph{Gauge transformation}:
\begin{equation}\label{gt} v(x,t):=u(x,t)\exp\left\{ik\int_{-\infty}^{x}\abs{u(y,t)}^2 \ dy,\right\}\end{equation}
where $k\in\mathbb{R}$. In \cite{Takaoka99}, Takaoka used the transformation with $k=-1$ to turn (\ref{dnls}) into
\begin{equation}
\begin{cases}
i\partial_t v+\Delta v = -iv^2\partial_x\overline{v}-\frac{1}{2}\abs{v}^4v \\
v(x,0)=v_0\in H^s(\mathbb{R}), s\geq \frac{1}{2}.
\end{cases}
\end{equation}
Note that the transformation replaces the term $\abs{u}^2\partial_x u$ with $v^2\partial_v \overline{u}$ which can be treated using the Fourier restriction norm method developed in \cite{Bour01}. In contrast to this type of proofs, we managed to get the local well-posedness of \eqref{dnls} (as a part of \Cref{thmm}) without using a gauge transformation. The advantage is that the idea can be easily generalized to get similar result for equation \eqref{dnls1} \\
\\
The paper is organized as follows. In the next subsection, we introduce some notations that are used in this paper. In section \ref{sec2}, we mention several linear and smoothing estimates and prove the maximal function estimate and bilinear estimate. In section \ref{sec3}, we introduce the solution space $X_N$ and nonlinear space $Y_N$ for functions supported at frequency $N$ and prove the main linear and bilinear estimate for functions in these spaces using a solution decomposition technique from \cite{BeIoKeTa}. In section \ref{sec4}, we prove a multilinear estimate. Having all the ingredients that we need, we finish the proof of \Cref{thmm1} in the same section. For \Cref{thmm}, we divide the proof into different sections by the degree $d$ of $P(u,\overline{u})$. In section \ref{sec6}, we prove \Cref{thmm} in the case of $d\geq 6$ . Since the case $d=5$ requires some frequency-modulation analysis, we will introduce the notion of $X^{s,b}$ space along with several well-known estimates in section \ref{sec7}, and use these results to conclude the proof of \Cref{thmm} in the same section. Finally, we prove another multilinear estimate and use it to finish the proof of \Cref{gwp1} in \Cref{proofgwp}.\\

\noindent \textbf{Notations.} The following notations will be used for the rest of the paper. For $1\leq p,q \leq \infty$, we use $\abss{f}_{L^p}$ to denote the $L^p$ norm, and we define the mixed norm
\[\abss{f}_{L_x^pL_t^q}:= \big\|{\abss{f(x,t)}_{L_t^q(I)}}\big\|_{L_x^p(\mathbb{R})},\]
where $I=[-1,1]$ if $d=3,4$ and $I=\mathbb{R}$ if $d\geq 5$. The norm $\abss{f}_{L_t^pL_x^q}$ is defined similarly. We define the Fourier transform and the inverse Fourier transform of $f(x)$ by
\begin{equation*}
\begin{split}
\hat{f}(\xi)&:=\frac{1}{\sqrt{2\pi}}\int_{\mathbb{R}} e^{-ix\xi}f(x) \ dx, \\
\check{f}(x)&:= \frac{1}{\sqrt{2\pi}}\int_{\mathbb{R}} e^{ix\xi}f(\xi) \ d\xi.
\end{split}
\end{equation*}
To simplify the proofs, we will always drop the constant $\frac{1}{\sqrt{2\pi}}$ from these transforms. For $s\in \mathbb{R}$, we denote by $D^s=(-\Delta)^{s/2}$ the Riesz potential of order $-s$. The Sobolev space $H_x^s$ is defined by the norm \[\abss{u}_{H_x^s}:=\abss{(1+\xi^2)^{\frac{s}{2}}\widehat{u}(\xi)}_{L^2_{\xi}}.\]
The Banach space of bounded $H_x^s$-valued continuous functions is denoted by
\[ C_t^0H^s_x(I\times J) := \left\{ f\in C(I; H^s_x(J)): \sup_{t\in I} \abss{f(x, t)}_{H^s_x(J)}<\infty \right\}.   \]
Let $u\in L^2_x$. We define the Schr{\"o}dinger propagator by
\[e^{it\Delta}u(x,t) := \int_{\mathbb{R}}e^{ix\xi-it\xi^2}\hat{u} \ d\xi. \]
The notation $a\lesssim b$ and $a\sim b$ means $a\leq Cb$ and $ca \leq b \leq CA$, respectively, for some positive constants $c$ and $C$, which depend on $P(z)$ but not on the functions involved in these estimates. \\
\noindent
We frequently split the frequency space into dyadic intervals, so whenever $M$ and $N$ is mentioned, we assume that $M,N\in 2^{\mathbb{Z}}$. Let $\psi(\xi)$ be a smooth cutoff function supported in $ \abs{\xi} \leq 4$ and equal $1$ on $\abs{\xi} \leq 2$. We define $\psi_N=\psi\left(\frac{\xi}{N}\right)-\psi\left(\frac{2\xi}{N}\right)$. Denote by $P_N$ the Littlewood-Paley projection at frequency $N$, that is
\[\widehat{P_Nf}(\xi)=\psi_N(\xi)\hat{f}(\xi)\]
Define $P_{\leq N}$ and $P_{> N}$ to be the projections of frequency less than and greater than $N$:
\begin{equation*}
\begin{split}
\widehat{P_{\leq N}f}(\xi)&= \psi_{\leq N}\hat{f}(\xi):=\sum_{M\leq N}\psi_M(\xi)\hat{f}(\xi), \\
\widehat{P_{> N}f}(\xi)&=\psi_{> N}\hat{f}(\xi):=\sum_{M> N}\psi_M(\xi)\hat{f}(\xi).
\end{split}
\end{equation*}
We will sometimes shorten the notation by $f_N:=P_Nf$. For $s\geq 0$, we can define the space $H^s$  and the homogeneous Sobolev space $\dot{H}^s$ using the Littlewood-Paley projections
\begin{equation*}
\begin{split}
\abss{u}_{\dot{H}^s}&:= \Big(\sum_{N_i\in 2^{\mathbb{Z}}}N_i^{2s}\abss{P_{N_i}u}^2_{L^2}\Big)^{\frac{1}{2}} \\
\abss{u}_{H^s} &:= \abss{P_{\leq 1}u}_{L^2}+\Big(\sum_{N_i\in 2^{\mathbb{N}}}N_i^{2s}\abss{P_{N_i}u}^2_{L^2}\Big)^{\frac{1}{2}}.
\end{split}
\end{equation*}

\section{Preliminary Results}\label{sec2}
\subsection{Bernstein type inequality}
\noindent
We begin with the Bernstein inequality for the Littlewood-Paley projections. Note that this is different from the standard result in literatures which is the same estimate but for the space $L_t^qL_x^p$.
\begin{lemma}
	For any pair of $1\leq p,q \leq\infty$, we have
	\begin{equation}\label{p1}
	\abss{\partial_x P_Nf}_{L_x^pL_t^q} \lesssim N\abss{P_Nf}_{L_x^pL_t^q}
	\end{equation}
\end{lemma}
\begin{proof}
	Let $\tilde{P}_N:=P_{N/2}+P_N+P_{2N}$ be a Littlewood-Paley projection at a wider frequency interval with corresponding multiplier $\widetilde{\psi}_N$. We can rewrite the term on the left-hand side as
	\[  \partial_x \tilde{P}_NP_Nf =(\partial_x \widecheck{\widetilde{\psi}}_N)* P_Nf(x,t). \]
	For each $x$, we have an inequality
	\[  \abss{\partial_x P_Nf}_{L_t^q} \leq \abs{\partial_x \widecheck{\widetilde{\psi}}_N}* \abss{P_Nf(x,t)}_{L_t^q} . \]
	After taking the $L_x^p$ norm and apply Young's inequality, we have
	\[\abss{\partial_x P_Nf}_{L_x^pL_t^q} \leq \abss{\partial_x \widecheck{\widetilde{\psi}}_N}_{L_x^1} \abss{P_Nf}_{L_x^pL_t^q} \lesssim N\abss{P_Nf}_{L_x^pL_t^q}. \]
\end{proof}
\noindent This lemma helps us quantify derivatives of a function supported in a dyadic frequency interval, which will come in handy in the proofs of multilinear estimates in section \ref{sec4} - \ref{sec7}. 
\subsection{Stationary phase lemmas} 
\noindent
We mention here stationary phase results from harmonic analysis, which will be used in the next subsection. See \cite[p.331-334]{D} for their proofs.
\begin{lemma}\label{sta1}
	Suppose that $\phi$ and $\psi$ are smooth functions and $\psi$ is compactly supported in $(a,b)$. If $\phi'(\xi)\not= 0$ for all $\xi\in [a,b]$, then
	\[\left|\int_{a}^{b}e^{i\lambda\phi(\xi)}\psi(\xi) \ d\xi\right| \leq \frac{C}{\abs{\lambda}^k}\]
	for all $k\geq 0$, where the constant $C$ depends on $\phi,\psi$ and $k$.
\end{lemma}
\begin{lemma}\label{sta2}
	Suppose that $\psi: \mathbb{R} \to \mathbb{R}$ is smooth, $\phi$ is a real-valued $C^2$-function in $(a,b)$ and $\phi''(\xi)\gtrsim 1$. Then,
	\[\left|\int_{a}^{b}e^{i\lambda\phi(\xi)}\psi(\xi) \ d\xi\right| \lesssim \frac{1}{\abs{\lambda}^{\frac{1}{2}}}\left(\abs{\psi(b)}+\int_{a}^{b}\abs{\psi'(\xi)}\ d\xi\right).\]
\end{lemma}

\subsection{Strichartz and local smoothing estimates} 
\noindent
In our study, the nonlinear effect of the equation \eqref{dnls2} with small initial data $u_0$ plays a major role in the perturbative analysis. As we mentioned in section 1, the main difficulty is a lost of derivative in the nonlinearity. In this regard, we will need the Strichartz estimate for the Schr{\"o}dinger propagator and the smoothing estimate \eqref{stri2} which gives a $\frac{1}{2}$-order derivative gain of the linear solution in a suitable norm. We will also prove a maximal function type estimate \eqref{li1} which will be used for the analysis of the nonlinear term. 
\begin{prop}
	Let $f\in L^2$. Then, we have the following estimates
	\begin{align}
	\label{stri1}\abss{e^{it\Delta}f}_{L_t^qL^p_x} &\lesssim \abss{f}_{L_x^2},
	\intertext{where $\dfrac{2}{q}+\dfrac{1}{p}=\dfrac{1}{2}$ and $2\leq p \leq \infty$, and}
	\label{stri2}\abss{D^{\frac{1}{2}}e^{it\Delta}f}_{L_x^{\infty}L^2_t} &\lesssim \abss{f}_{L_x^2}. 
	\end{align}
\end{prop}
\begin{proof} The first inequality is the well-known Strichartz estimate. The proof can be found, for example, in \cite{A} and \cite{E}. The proof of \eqref{stri2} can be found in Theorem $4.1$ of \cite{B}. 
\end{proof}
\noindent
The following maximal function type estimate tells us that for the linear equation with time-and-frequency localized initial data in $H^s(\mathbb{R})$ where $s\geq \frac{1}{2}$, the solution is well-controlled in $L^{\gamma}_xL^{\infty}_t(\mathbb{R}\times I )$, where $I=[-1,1]$ when $\gamma=2,3$ and $I=\mathbb{R}$ when $\gamma\geq 4$.
\begin{prop}\label{thm2}
	Let $u\in L^2_x(\mathbb{R})$. 
	\begin{enumerate}
		\item
		If $\gamma=2$ or $3$, assume that $\emph{supp}(\abs{\hat{u}}) \subseteq [N,4N]$ where $N\in 2^{\mathbb{N}}$ or $\emph{supp}(\abs{\hat{u}}) \subseteq [0,1]$, in which case we consider $N=1$, then
		\begin{subequations}\label{li1}
			\begin{align}
			\label{li10}\abss{\chi_{[-1,1]}(t)e^{it\Delta}u(x)}_{L_x^{\gamma}L_t^{\infty}} &\lesssim N^{\frac{1}{\gamma}}\abss{u}_{L_x^2}, \\
			\intertext{\item If $\gamma\geq 4$, assume that $\emph{supp}(\abs{\hat{u}}) \subseteq [N,4N]$ where $N\in 2^{\mathbb{Z}}$, we have}
			\label{li11}\abss{e^{it\Delta}u(x)}_{L_x^{\gamma}L_t^{\infty}} &\lesssim N^{\frac{\gamma-2}{2\gamma}}\abss{u}_{L_x^2}.
			\end{align}
		\end{subequations}
	\end{enumerate}
\end{prop}
\noindent
\emph{Remark: } We see that the estimate \eqref{li10} is local in time while \eqref{li11} is global. By setting $\gamma = d-1$, this leads to the local and global results in \Cref{thmm1} and \Cref{thmm}.
\begin{proof}
	We refer to Theorem $2.5$ in \cite{B} for a proof of the case $\gamma=4$. Let $s_0=s_0(\gamma)=\frac{1}{\gamma}$ for $\gamma=2,3$ and $s_0=\frac{\gamma-2}{2\gamma}$ for $\gamma\geq 5$. We define an operator $T:L_x^2 \rightarrow L_x^{\gamma}L_t^{\infty}$ by $Tu=\chi_{[-1,1]}(t)e^{it\Delta}u$, yielding $T^*F=\int_{-1}^{1}e^{-it\Delta}F \ dt$. Using the $TT^*$ argument, it follows that \eqref{li1} is equivalent to either of the following estimates for $F\in L_x^2L_t^1(\mathbb{R}\times\mathbb{R})$ with the same frequency support as $u$ in the cases of $\gamma=2,3$.
	\begin{align}
	\left\|\int_{-1}^{1}e^{-it\Delta}F(x,t) \ dt\right\|_{L_x^2} &\lesssim N^{s_0}\abss{F}_{L_x^{\frac{\gamma}{\gamma-1}}L_t^1} \label{dual}\\
	\label{li2}\left\|\chi_{[-1,1]}(t)\int_{-1}^{1}e^{i(t-s)\Delta}F(x,s) \ ds\right\|_{L_x^{\gamma}L_t^{\infty}} &\lesssim N^{2s_0}\abss{F}_{L_x^{\frac{\gamma}{\gamma-1}}L_t^1}.
	\end{align}
	For $\gamma\geq 5$, we have the same estimates but with integrals on $\mathbb{R}$. Thus, it suffices to prove \eqref{li2}. First, we assume that $F\in \mathcal{S}(\mathbb{R})$. Since $F=P_{\leq 4N}F$, the inverse Fourier transform of $e^{i(t-s)\xi^2}\widehat{F}$ is defined by
	\begin{equation*}
	\begin{split}
	\mathcal{F}_x^{-1}\left(e^{i(t-s)\xi^2}\widehat{F}(\xi,s)\right)  &=c\int_{\mathbb{R}}e^{i(t-s)\xi^2+ix\xi}\widehat{F}(\xi,s) \ d\xi \\
	&=\mathcal{F}_x^{-1}\left(e^{-i(t-s)\xi^2}\psi\left(\frac{\xi}{4N}\right)\right)*F(x,s).
	\end{split}
	\end{equation*}
	Since $-1\leq t,s \leq 1$ implies $-2\leq t-s \leq 2$, the term on the right of \eqref{li2} can be replaced by
	\begin{equation*}
	\begin{split}
	&\int_{\mathbb{R}}\mathcal{F}_x^{-1}\left(\chi_{[-2,2]}(t-s)e^{-i(t-s)\xi^2}\psi\left(\frac{\xi}{4N}\right)\right)*F(x,s) \ ds \\
	&=\mathcal{F}_x^{-1}\left(\chi_{[-2,2]}(t)e^{-it\xi^2}\psi\left(\frac{\xi}{4N}\right)\right)\star F(x,t) \\
	&= c_1K_1\star F
	\end{split}
	\end{equation*}
	where $\star$ denotes the space-time convolution and
	\begin{equation}\label{non}K_1(x,t)=\int_{\mathbb{R}}e^{-it\xi^2+ix\xi}\chi_{[-2,2]}(t)\psi\left(\frac{\xi}{4N}\right) \ d\xi.\end{equation}
	Similarly, for $\gamma\geq 5$ we have
	\[\int_{\mathbb{R}}e^{i(t-s)\Delta}F(x,s) \ ds = c_2K_2\star F\]
	where
	\begin{equation}\label{non2}
	K_2(x,t)=\int_{\mathbb{R}}e^{-it\xi^2+ix\xi}\psi\left(\frac{\xi}{4N}\right) \ d\xi.
	\end{equation}
	To finish the proof, we need the following lemma.
	\begin{lemma}\label{lemma1}
		Let $K_1(x,t)$ and $K_2(x,t)$ be as in \eqref{non} and \eqref{non2}. Then, for $i=1,2$
		\begin{equation}\label{k}
		\abss{K_i}_{L_x^{\frac{\gamma}{2}}L_t^{\infty}}\lesssim N^{2s_0}.
		\end{equation}
	\end{lemma}
	\vspace{8pt}
	\noindent We continue the proof of \Cref{thm2}. By applying Young's inequality and \Cref{lemma1}, we obtain
	\[\abss{K_i\star F}_{L_x^{\gamma}L_t^{\infty}}\leq \abss{K_i}_{L_x^{\frac{\gamma}{2}}L_t^{\infty}}\abss{F}_{L_x^{\frac{\gamma}{\gamma-1}}L_t^1}\]
	as desired. We then finish the proof by the usual density argument.
\end{proof}
\begin{proof}[Proof of \Cref{lemma1}] Let $I=[-1,1]$ when $\gamma=2,3$ and $I=\mathbb{R}$ when $\gamma\geq 4$. We divide $\mathbb{R}\times I$ into three regions
	\begin{equation*}
	\begin{split}
	\Omega_1 &:= \{(x,t)\in \mathbb{R}\times I \ | \ \abs{x}\leq \frac{1}{N}\} \\
	\Omega_2 &:= \{(x,t)\in \mathbb{R}\times I \ | \ \abs{x}\geq 64N\abs{t} \ , \ \abs{x}>\frac{1}{N}\} \\
	\Omega_3 &:= \{(x,t)\in \mathbb{R}\times I \ | \ \abs{x}< 64N\abs{t} \ , \ \abs{x}>\frac{1}{N}\},
	\end{split}
	\end{equation*}
	and we will estimate $K_i(x,t)$ in each region. For a fixed $x\in\mathbb{R}$ and $1\leq i \leq 3$, we define $\Omega_{x,i} := \{t\in I \ | \ (x,t)\in \Omega_i \}$. We consider the following two cases of values of $\gamma$. \\
	\begin{enumerate}[wide, labelwidth=!, labelindent=0pt]
		\item[\textbf{Case 1: }] $\gamma=2,3$.
		Note that in this case we always assume that $N\geq 1$. By a change of variable $\eta=\frac{\xi}{4N}$, we obtain
		\[K_1(x,t)=N\int_{\mathbb{R}}\chi_{[-2,2]}e^{-i16tN^2\eta^2+i4xN\eta}\psi(\eta) \ d\eta\]
		A simple estimate on $\Omega_1$ shows that
		\begin{equation}\label{li3}
		\begin{split}
		\int_{\abs{x}\leq \frac{1}{N}} \abs{K_{1}(x,t)}^{\frac{\gamma}{2}} \ dx \lesssim \frac{1}{N}\cdot N^{\frac{\gamma}{2}}\Big(\int_{\mathbb{R}} \psi(\eta) \ d\eta \Big)^{\frac{\gamma}{2}} \sim N^{\frac{\gamma-2}{2}}\leq N.
		\end{split}
		\end{equation}
		Next we consider the norm on $\Omega_2$. Note that the integrand in $K_{1}$ vanishes if $\abs{\eta} \geq 4$. Factoring out $-i16tN^2\eta^2+i4xN\eta=-i4xN(\eta-\frac{4tN}{x}\eta^2):=-ixN\phi_1(\eta)$ yields \[\abs{\phi'_1(\eta)}=\abs{1-8\frac{tN}{x}\eta}\geq 1-32\left|\frac{tN}{x}\right| \geq 1-32\cdot\frac{1}{64}=\frac{1}{2},\]
		for any $t\in \Omega_{x,2}$. Therefore, $\phi_1$ has no critical point in this region. By \Cref{sta1}, the integral in $K_{1}$ is bounded by $\abs{Nx}^{-k}$ for all $k\geq 0$. In particular, by choosing $k=2$, we obtain $\abs{K_{1}(x,t)} \lesssim N(N\abs{x})^{-2}=N^{-1}\abs{x}^{-2}$. We finish by computing the $L_x^{\frac{\gamma}{2}}L_t^{\infty}$ norm on $\Omega_2$:
		\begin{equation}\label{li4}
		\begin{split}
		\int \sup_{t\in\Omega_{x,2}}\abs{K_{1}(x,t)}^{\frac{\gamma}{2}} \ dx \lesssim N^{(\gamma-1)-\frac{\gamma}{2}} = N^{\frac{\gamma-2}{2}} \leq N.
		\end{split}
		\end{equation}
		\noindent Now we consider the norm on $\Omega_3$. Factoring out the exponential term $-i16tN^2\eta^2+i4xN\eta=-i4tN^2(4\eta^2-\dfrac{x\eta}{Nt}):=i4tN^2\phi_2(\eta)$ yields $\phi''_2(\eta)\gtrsim 1$, so we can apply \Cref{sta2} to $K_{1}$.
		\begin{equation}\label{li6}
		\begin{split}
		\left|K_{1}(x,t)\right| &= N\left|\int_{\mathbb{R}}e^{-itN^2\eta^2+ixN\eta}\psi(\eta) \ d\eta\right| \\ &\lesssim N\cdot \frac{1}{N\abs{t}^{\frac{1}{2}}} \\
		&< \frac{64N^{\frac{1}{2}}}{\abs{x}^{\frac{1}{2}}}.
		\end{split}
		\end{equation}
		Now we compute the $L_x^{\frac{\gamma}{2}}L_t^{\infty}$ norm of $K_{1}$. Observe that the finite time restriction yields $\abs{x} \lesssim N\abs{t} \leq 2N$ on $\Omega_3$. Therefore,
		\begin{equation}\label{li5}
		\begin{split}
            \int \sup_{t\in\Omega_{x,3}}\abs{K_1(x,t)}^{\frac{\gamma}{2}} \ dx \lesssim \int_{\abs{x}<64N\abs{t}} N^{\frac{\gamma}{4}}\abs{x}^{-\frac{\gamma}{4}} \ dx \lesssim N^{\frac{\gamma}{4} -\frac{\gamma-4}{4}}=  N.
		\end{split}
		\end{equation}
		Combining \eqref{li3},\eqref{li4} and \eqref{li5}, we have that
		\[ 	\abss{K_1}_{L_x^{\frac{\gamma}{2}}L_t^{\infty}} \lesssim N^{\frac{2}{\gamma}}.\]
		\item[\textbf{Case 2:} ]$\gamma \geq 5$. Since the estimates in \eqref{li3} and \eqref{li4} do not require any time restriction, we get the same results for $K_{2}$.
		\begin{equation}\label{lii1}
		\int_{\Omega_1\cup \Omega_2} \abs{K_{2}}^{\frac{\gamma}{2}} \ dx \lesssim  N^{\frac{\gamma-2}{2}}.
		\end{equation}	
		On $\Omega_3$, we have the same estimate as in \eqref{li6} for $K_{2}$. From the fact that $\abs{x}>\frac{1}{N}$ in this region, we have
		\begin{equation}\label{lii2}
		\int \sup_{t\in\Omega_{x,3}} \abs{K_2(x,t)}^{\frac{\gamma}{2}} \ dx\lesssim \int_{\abs{x}>\frac{1}{N}} N^{\frac{\gamma}{4}}\abs{x}^{-\frac{\gamma}{4}} \ dx \lesssim N^{\frac{\gamma}{4} +\frac{\gamma-4}{4}}=  N^{\frac{\gamma-2}{2}}.
		\end{equation}
		Note that we did not use the finite time restriction in this case. Combining \eqref{lii1} and \eqref{lii2}, we have that
		\[ 	\abss{K_2}_{L_x^{\frac{\gamma}{2}}L_t^{\infty}} \lesssim N^{\frac{\gamma-2}{\gamma}}.\]
	\end{enumerate}
\end{proof}
\noindent To estimate a product of functions as seen in the nonlinearity of DNLS, one usually employs the bilinear estimate which splits the product into estimating individual functions (see \cite{Bour02} where Bourgain proved the estimate in two dimensions).
\begin{theorem}[Bilinear Strichartz Estimate] \label{bi}
	For any $u,v\in L^2_x$, we have
	\begin{align}  						\abss{P_{\lambda}(e^{it\Delta}u\overline{e^{it\Delta}v})}_{L^2_{x,t}} &\lesssim \lambda^{-\frac{1}{2}}\abss{u}_{L^2}\abss{v}_{L^2} \label{bi15}
        \intertext{In addition, if $\hat{u}$ and $\hat{v}$ have disjoint supports and $\alpha:=\inf\abs{ \supp(\hat{u})-\supp(\hat{v})}$ is strictly positive, then we have}
	\abss{e^{it\Delta}ue^{it\Delta}v}_{L^{2}_{x,t}} &\lesssim \alpha^{-\frac{1}{2}}\abss{u}_{L^2}\abss{v}_{L^2}.\label{bil}
	\end{align}
\end{theorem}
\begin{proof}
	We follow the proof in \cite[Theorem 2.9]{F}. 
	By duality, this is equivalent to showing that for any $F\in C_c^{\infty}$,  $$\Big|\int F(\xi-\eta,\xi^2-\eta^2)\psi_{>\lambda}(\xi-\eta)\hat{u}(\xi)\bar{\hat{v}}(\eta) \ d\xi d \eta\Big| \lesssim \lambda^{-\frac{1}{2}}\abss{F}_{L^2_{\xi,\tau}}\abss{\hat{u}}_{L^2_{\xi}}\abss{\hat{v}}_{L^2_{\xi}}.$$
	For each fixed $\alpha$ and $\beta$, let $(\xi_{\alpha\beta}, \eta_{\alpha\beta})$ be a solution to $\alpha = \xi^2-\eta^2$ and $\beta = \xi-\eta$. We see that the change of variables $(\xi,\eta) \mapsto (\alpha,\beta)$ gives the Jacobian $J = 2(\eta-\xi)$. This together with Cauchy-Schwarz yield
	\begin{align*}
	\Big|\int F(\xi-\eta&,\xi^2-\eta^2)\psi_{>\lambda}(\xi-\eta)\hat{u}(\xi)\bar{\hat{v}}(\eta) \ d\xi d \eta\Big| \\
	&= \Big|\int F(\alpha,\beta)\psi_{>\lambda}(\beta)\hat{u}(\xi_{\alpha\beta})\bar{\hat{v}}(\eta_{\alpha\beta}) \frac{1}{J}\ d\alpha d \beta\Big| \\
	&\leq \abss{F}_{L^2_{\xi,\tau}}\Big(\int \abs{\psi_{>\lambda}(\beta)}^2\abs{\hat{u}(\xi_{\alpha\beta})}^2\abs{\hat{v}(\eta_{\alpha\beta})}^2 \frac{1}{J^2} \ d\alpha d \beta\Big)^{\frac{1}{2}} \\
	&= \abss{F}_{L^2_{\xi,\tau}}\left(\int \abs{\psi_{>\lambda}(\xi-\eta)}^2\abs{\hat{u}(\xi)}^2\abs{\hat{v}(\eta)}^2 \frac{1}{J} \ d\xi d \eta\right)^{\frac{1}{2}} \\
	&\lesssim \lambda^{-\frac{1}{2}}\abss{F}_{L^2_{\xi,\tau}}\abss{\hat{u}}_{L^2_{\xi}}\abss{\hat{v}}_{L^2_{\xi}}.
	\end{align*}
	This concludes the proof of \eqref{bi15}. The proof for \eqref{bil} is essentially the same, but $\xi-\eta$ is replaced by $\xi+\eta$, $\xi^2-\eta^2$ is replaced by $\xi^2+\eta^2$ and there is no $\psi_{>\lambda}$. The conclusion follows from the observation that 
	\[\frac{1}{\abs{J}}=\frac{1}{2\abs{\eta-\xi}}\gtrsim \frac{1}{\alpha}.\]
\end{proof}
\noindent
We will need a variant of this estimate adapted to the $X^s$ space \eqref{bi2} for our trilinear estimate \eqref{linn6}. The details will be explained in the next section.

\section{The main linear estimate}\label{sec3}
\noindent In this section, we consider a nonlinear Schr{\"o}dinger equation
\begin{equation}
\begin{split}\label{eq}
iu_t +\Delta u &= F \\
u(x,0) &= u_0.
\end{split}
\end{equation}
Let $I=[-1,1]$ if $d=3,4$ and $I=\mathbb{R}$ if $d\geq 5$. A solution $u(x,t)\in \mathbb{R}\times I$ can be represented by the Duhamel formula
\begin{equation}u(x,t)=e^{it\Delta}u_0-i\int_0^t e^{i(t-s)\Delta} F(s) \ ds.\end{equation}
\noindent
In the proof of \Cref{thmm1} and \Cref{thmm}, the spaces that we use are based on the following norms which take a function $u$ supported at dyadic frequency interval $\sim N$.
\begin{equation}
\begin{split}\label{norm}
\abss{u}_{Y_N} &= \inf\{N^{-\frac{1}{2}}\abss{u_1}_{L_x^1L_t^2} +\abss{u_2}_{L_{t}^1L_x^2}\ |\ u_1+u_2=u\}\\
\abss{u}_{X_N} & =\abss{u}_{L_t^{\infty}L_x^2}+N^{-s_0}\abss{u}_{L_x^{d-1}L_t^{\infty}}+N^{\frac{1}{2}}\abss{u}_{L_x^{\infty}L_t^2} \\
& \ \ \ +N^{-\frac{1}{2}}\abss{(i\partial_t+\Delta)u}_{Y_N} ,
\end{split}
\end{equation}
where $L_t^{\infty}L_x^2=L_t^{\infty}L_x^2(I\times \mathbb{R})$ and $L_x^{p}L_t^{q}=L_x^{p}L_t^{q}(\mathbb{R} \times I)$. These norms satisfy the following linear estimate, which makes them suitable for the contraction argument.
\begin{theorem}
	Let $u$ be a solution to equation \eqref{eq}. Then,
	\begin{equation}\label{main}
	\abss{P_Nu}_{X_N} \lesssim \abss{u_0}_{L^2_x}+\abss{P_NF}_{Y_N}.
	\end{equation}
\end{theorem}
\noindent This immediately follows from the Duhamel formula and the following three propositions.
\begin{prop} For any $u_0\in L^2_x(\mathbb{R})$, we have
	\begin{equation}\label{linear}\abss{e^{it\Delta}P_Nu_0}_{X_N} \lesssim \abss{u_0}_{L^2_x}. \end{equation}
\end{prop}
\begin{proof}
	This follows from the Strichartz estimate \eqref{stri1}, the smoothing estimate \eqref{stri2} and \eqref{li10} if $d=3,4$ or \eqref{li11} if $d\geq 5$.
\end{proof}
\begin{prop}\label{prop1} For any function $F(x,t)$ such that $P_NF \in L^1_xL^2_t$, we have
	\begin{equation}\label{est1}
	\left\|\int_0^t e^{i(t-s)\Delta} P_NF(s) \ ds\right\|_{X_N} \lesssim \abss{P_NF}_{Y_N}.
	\end{equation}
\end{prop}
\begin{proof}
	It follows from Minkowski inequality and \eqref{linear} that
	\begin{equation*}
	\begin{split}
	\left\|\int_0^t e^{i(t-s)\Delta} P_NF(s) \ ds\right\|_{X_N} &\leq \int_{\mathbb{R}}\abss{e^{i(t-s)\Delta}P_NF(s)}_{X_N} \ ds \\
	&\lesssim \int_{\mathbb{R}}\abss{P_NF(s)}_{L^2_x} \ ds \\
	&= \abss{P_NF}_{L^1_tL^2_x}.
	\end{split}
	\end{equation*}
	Therefore, it suffices to prove that
	\begin{equation}\label{est0}
	\left\|\int_0^t e^{i(t-s)\Delta} P_NF(s) \ ds\right\|_{X_N} \lesssim N^{-\frac{1}{2}}\abss{P_NF}_{L^1_xL^2_t}.
	\end{equation}
	Let $K_0$ be the fundamental solution of Schr{\"o}dinger equation i.e.
	\[K_0(x,t)=\mathcal{F}^{-1}(e^{-it\xi^2})=\frac{1}{\sqrt{4\pi i t}}e^{ix^2/4t}.\]
	Thus,
	\begin{equation}\label{inhom}
	\begin{split}\int_0^t e^{i(t-s)\Delta} P_NF(x,s) \ ds &= \int_0^t\int_{\mathbb{R}} P_N\left[K_0(x-y,t-s)F(y,s)\right] \ dyds \\
	&= \int_{\mathbb{R}}\int_0^t P_N\left[K_0(x-y,t-s)F(y,s)\right] \ dsdy \\
	&:= \int_{\mathbb{R}} w_y \ dy,
	\end{split}
	\end{equation}
	In order to proceed, we will make use of the following lemma.
	\begin{lemma}\label{lem1}
		For any $N\in 2^{\mathbb{Z}}$, the function $w_y$ defined in \eqref{inhom} satisfies the following estimate:
		\begin{equation}\label{est2}
		\abss{w_y}_{X_N} \lesssim N^{-\frac{1}{2}}\abss{F(y)}_{L_t^2}.
		\end{equation}
	\end{lemma}
	\noindent Continuing the proof of \Cref{prop1}, we see that the estimate \eqref{est0} follows immediately from \eqref{est2}. \end{proof}
\begin{proof}[Proof of \Cref{lem1}]
	By translation invariance, it suffices to assume that $y=0$. Denote $F_0(t):=F(0,t)$. To proceed, we use the following decomposition which was first introduced in \cite{BeIoKeTa} to deal with Scr{\"o}dinger maps.
	\begin{equation}\label{dec2}
	w_0(x,t) = -e^{it\Delta}\mathcal{L}v_0(x)-(P_{<N/2^{50}}1_{x>0})e^{it\Delta}v_0(x)+h(x,t),
	\end{equation}
	where $\mathcal{L}:L^2_x(\mathbb{R})\to L^2_x(\mathbb{R})$ is an operator and
	\begin{equation}\label{dec1}
	\abss{\mathcal{L}v_0}_{L_x^2}+\abss{v_0}_{L_x^2}+N^{-1}(\abss{\Delta h}_{L_{x,t}^2}+\abss{ h_t}_{L_{x,t}^2}) \lesssim N^{-\frac{1}{2}}\abss{F_0}_{L_t^2}.
	\end{equation}
	To prove the claim, first we rewrite the definition of $w_0$ as
	\begin{equation}\label{PN}
	\begin{split}
	w_0(x,t) &=\int_{\mathbb{R}} \chi_{[0,\infty)}(t-s)P_N[K_0(x,t-s)]F_0(s) \ ds \\
	& \ \ \ -e^{it\Delta}\int_{-\infty}^{0}P_N[K_0(x,-s)]F_0(s) \ ds \\
	&=(\chi_{[0,\infty)}P_NK_0)\ast_t F_0-e^{it\Delta}\int_{-\infty}^{0}P_NK_0(x,-s)F_0(s) \ ds ,
	\end{split}
	\end{equation}
	where $\ast_t$ is the time convolution. The space-time Fourier transform of the first term is equal to
	\begin{equation}\label{dec5}
	\frac{\psi_N(\xi)}{-\tau-\xi^2-i0}\widehat{F}_0(\tau),
	\end{equation}
	where $\widehat{F}_0$ is the time Fourier transform of $F_0$. We define 
	\begin{equation}\label{v0}
	\hat{v}_0(\xi) := \psi_N(\xi)\widehat{F}_0(-\xi^2).
	\end{equation}
	We see that $v_0$ is supported at frequency $\sim N$. By changing variables we obtain the following estimate,
	\begin{equation}\label{dec3}
	\abss{v_0}_{L^2_x} \lesssim N^{-\frac{1}{2}}\abss{F_0}_{L^2_t}.
	\end{equation}
	We apply the spatial Fourier transform to the second term 
	\begin{equation}\label{L}
	\begin{split}
	\int_{-\infty}^{0}\widehat{P_NK}_0(x,-s)F_0(s) \ ds &=\psi_{N}(\xi)\int_{-\infty}^{0}e^{is\xi^2}F_0(s) \ ds \\
	&= \psi_{N}(\xi)\mathcal{F}_t(\chi_{(0,\infty]} F_0) (-\xi^2) \\
	&:= \widehat{\mathcal{L}v_0}(\xi).
	\end{split}
	\end{equation}
	We see that $\mathcal{L}v_0$ is supported at frequency $\sim N$. It follows from a change of variables that
	\[\abss{\mathcal{L}v_0}_{L^2_x}\lesssim N^{-\frac{1}{2}}\abss{F_0}_{L^2_t}.\]
	Applying the Fourier transform to $e^{it\Delta}v_0$,
	\[\mathcal{F}(e^{it\Delta}v_0)=\psi_N(\xi)\widehat{F}_0(-\xi^2)\mathcal{F}_t (e^{-it\xi^2})=\psi_N(\xi)\widehat{F}_0(-\xi^2)\delta_{\tau+\xi^2}.\]
	Assume that $\xi > 0 $ and consider the distribution $\delta_{\tau+\xi^2}$. For any $\phi \in \mathcal{S}(\mathbb{R} \times \mathbb{R})$, by a change of variables
	\[\int_{0}^{\infty}\phi(\xi,-\xi^2) \ d\xi=\int_{-\infty}^{0}\frac{1}{2\sqrt{-\tau}}\phi(\sqrt{-\tau},\tau) \ d\tau. \]
	Thus, $1_{\xi>0}\delta_{\tau+\xi^2}=1_{\tau<0}\frac{1}{2\sqrt{-\tau}}\delta_{\xi-\sqrt{-\tau}}$. Therefore, the following computation holds.
	\begin{align*}
        \mathcal{F}\big\{(P_{<N/2^{50}}& 1_{x>0})e^{it\Delta}v_0\big\}(\xi,\tau) \\ &=(\psi_N(\xi)\widehat{F}_0(-\xi^2)\delta_{\tau+\xi^2})* \frac{\psi_{<N/2^{50}}(\xi)}{\xi+i0} \\
	&=\left(\frac{\psi_N(\xi)}{2\sqrt{-\tau}}\widehat{F}_0(-\xi^2)\delta_{\xi-\sqrt{-\tau}}\right)* \frac{\psi_{<N/2^{50}}(\xi)}{\xi+i0}  \\
	&=\frac{\psi_N(\sqrt{-\tau})\widehat{F}_0(\tau)}{2\sqrt{-\tau}}\frac{\psi_{<N/2^{50}}(\xi-\sqrt{-\tau})}{\xi-\sqrt{-\tau}+i0} \\
	&= \psi_N(\sqrt{-\tau})\psi_{<N/2^{50}}(\xi-\sqrt{-\tau})\widehat{F}_0(\tau)\frac{\xi+\sqrt{-\tau}}{2\sqrt{-\tau}}\frac{1}{\xi ^2+\tau+i0}. 
	\end{align*}
	With this and \eqref{dec5}, the space-time Fourier transform of the remainder term is given by
	\begin{align*}
	\hat{h}(\xi,\tau) &= \left(\psi_N(\xi)-\psi_N(\sqrt{-\tau})\psi_{<N/2^{50}}(\xi-\sqrt{-\tau})\frac{\xi+\sqrt{-\tau}}{2\sqrt{-\tau}}\right)\frac{\widehat{F}_0(\tau)}{-\xi^2-\tau-i0} \\
                      &:= A(\xi,\tau)\widehat{F}_0(\tau).\numberthis \label{hh}
	\end{align*}
	The term in the bracket is bounded, supported in $\{0<\xi\sim N\}$ and vanishes when $\xi =\sqrt{-\tau}$, canceling out the singularity. Since the same result holds for $\xi<0$, this implies that
	\begin{equation}\label{dec4}
	\abss{\Delta h}_{L^2_{x,t}}+\abss{\partial_t h}_{L^2_{x,t}} \sim \abss{(\xi^2+\abs{\tau})\hat{h}}_{L^2_{\xi,\tau}} \lesssim N^{\frac{1}{2}}\abss{\widehat{F}_0(\tau)}_{L^2_{\tau}}.
	\end{equation}
	The estimate \eqref{dec1} then follows from \eqref{dec3} and \eqref{dec4}. \\
	\\
	\textbf{Remark: }It is important to note that $v_0, Lv_0$ and $h$	are supported at frequency $\sim N$, since we will need this fact in any proof that employ the decomposition \eqref{dec2}.  \\
	\\
	\noindent We are now ready to prove \eqref{est2}. By Bernstein's inequality and direct $L^2$ integration on $A(\xi,\tau)$,
	\begin{equation*}
	\begin{split}
	\abss{h}_{L_x^{d-1}L_t^{\infty}} \leq \abss{\mathcal{F}_t h}_{L_x^{d-1}L_{\tau}^{1}} 
	&\lesssim \abss{\mathcal{F}_t h}_{L_{\tau}^{1}L^{d-1}_x}\\
	&\lesssim N^{\frac{d-3}{2(d-1)}}\abss{\mathcal{F}_t h}_{L_{\tau}^{1}L^{2}_x}	\\
	&= N^{\frac{d-3}{2(d-1)}}\abss{\hat{h}}_{L_{\tau}^{1}L^{2}_{\xi}} \\
	&\leq N^{\frac{d-3}{2(d-1)}}\abss{A(\xi,\tau)}_{L^2_{\tau,\xi}}\abss{\widehat{F}_0(\tau)}_{L^2_{\tau}},
	\end{split}
	\end{equation*}
	where $A(\xi,\tau)$ is defined as in \eqref{hh} when $\xi>0$,. We split the integral in $	\abss{A(\xi,\tau)}_{L^2_{\tau,\xi}}^2$ as
	\begin{equation*}
	\begin{split}
	\abss{A(\xi,\tau)}_{L^2_{\tau,\xi}}^2 &= \int_{\abs{\xi-\sqrt{-\tau}}<\frac{N}{2^{100}}} \abs{A(\xi,\tau)}^2\ d\xi d\tau \\
	& \ \ \ +\int_{\abs{\xi-\sqrt{-\tau}}\geq\frac{N}{2^{100}}} \abs{A(\xi,\tau)}^2\ d\xi d\tau \\
	&:= A_1 +A_2.
	\end{split}
	\end{equation*}
	Note that $\psi_N(\xi)=\psi_N(\sqrt{-\tau})+(\xi-\sqrt{-\tau})O(\frac{1}{N})$ as $\xi \to \sqrt{-\tau} $. If $\abs{\xi-\sqrt{-\tau}}<\frac{N}{2^{100}}$, then $\psi_{<N/2^{50}}(\xi-\sqrt{-\tau})=1$ and it follows that
	\begin{equation*}
	\begin{split}
	\psi_N(\xi)-\psi_N(\sqrt{-\tau})&\psi_{<N/2^{50}}(\xi-\sqrt{-\tau})\frac{\xi+\sqrt{-\tau}}{2\sqrt{-\tau}} \\ &=\frac{\psi_N(\sqrt{-\tau})(\sqrt{-\tau}-\xi)}{2\sqrt{-\tau}}+(\xi-\sqrt{-\tau})O\Big(\frac{1}{N}\Big).
	\end{split}
	\end{equation*}
	Since $A(\xi,\tau)$ is supported in the region $\xi\sim N$, we have that
	\begin{equation*}
	\begin{split}
	A_1 &\lesssim \int_{\tau\sim -N^2}\int_{\xi\sim N} \frac{1}{-2\tau(\xi+\sqrt{-\tau})^2 }+\frac{1}{N^2(\xi+\sqrt{-\tau})^2}\ d\xi d\tau \lesssim \frac{1}{N}.
	\end{split}
	\end{equation*}
	On the other hand, under the assumptions that, $\xi\sim N$ and $\abs{\xi-\sqrt{-\tau}}\geq\frac{N}{2^{100}}$, we have  $\abs{\xi^2+\tau}=\abs{(\xi+\sqrt{-\tau})(\xi-\sqrt{-\tau})}\gtrsim \frac{N^2}{2^{100}}$. Thus, by a change of variables $(\xi,\tau)\mapsto (\xi,\eta)$ where $\eta := \tau+\xi^2$, we have 
	\begin{equation*}
	\begin{split}
	A_2 &\leq \int_{-\infty}^{0}\int_{\abs{\xi-\sqrt{-\tau}}\geq\frac{N}{2^{100}}} \frac{\psi_{N}(\xi)}{(\xi^2+\tau)^2 } +\frac{\psi_{N}(\sqrt{-\tau})\psi_{<N/2^{50}}(\xi-\sqrt{-\tau})}{-4\tau (\xi+\sqrt{-\tau})^2}\ d\xi d\tau \\ &\lesssim \int_{\xi\sim N}\int_{\abs{\eta}\gtrsim \frac{N^2}{2^{100}}} \frac{1}{\eta^2} \ d\eta d\xi+\int_{\tau\sim -N^2}\int_{\xi\sim N}\frac{1}{-4\tau (\xi+\sqrt{-\tau})^2}\ d\xi d\tau \\
	&\lesssim \int_{\xi\sim N} \frac{1}{N^2}\ d\xi +\frac{1}{N}\\
	&\lesssim \frac{1}{N},
	\end{split}
	\end{equation*}
	and a similar result holds when $\xi<0$. From this, we can conclude that
	\begin{equation}\label{h2}
	\abss{h}_{L_x^{d-1}L_t^{\infty}}\lesssim N^{\frac{d-3}{2(d-1)}}\abss{A(\xi,\tau)}_{L^2_{\tau,\xi}}\abss{\widehat{F}_0(\tau)}_{L^2_{\tau}} \lesssim N^{\frac{d-3}{2(d-1)}-\frac{1}{2}}\abss{F(0)}_{L^2_{t}}.
	\end{equation}
	Similarly, we have the following,
	\begin{equation}\label{inf}
	\abss{h}_{L_{x,t}^{\infty}} \lesssim \abss{F(0)}_{L^2_{t}}.
	\end{equation}		
	In particular, for $d=3$ and $N\geq 1$, we have that
	\begin{equation}\label{h4}
	N^{-\frac{1}{2}}\abss{h}_{L_x^2L_t^{\infty}} \leq \abss{h}_{L_x^{2}L_t^{\infty}} \lesssim N^{-\frac{1}{2}}\abss{F(0)}_{L_t^2}.
	\end{equation}
	Similarly, by Sobolev's embedding,
	\begin{equation}\label{h5}
	N^{\frac{1}{2}}\abss{h}_{L_x^{\infty}L_t^2}\lesssim N^{\frac{1}{2}}\abss{h}_{L^2_tL^{\infty}_x} \lesssim  N^{-1}\abss{\Delta h}_{L^2_{x,t}}\lesssim N^{-\frac{1}{2}}\abss{F(0)}_{L^2_{t}}.
	\end{equation}
	where we used \eqref{dec1} in the last step. Lastly, it follows from \eqref{h2} that
	\begin{equation}\label{h6}
	\abss{h}_{L_t^{\infty}L_x^2} \leq \abss{h}_{L_x^2L_t^{\infty}} \lesssim N^{-\frac{1}{2}}\abss{F(0)}_{L^2_{t}}.
	\end{equation}
	Putting together \eqref{h2}, \eqref{h5} and \eqref{h6}, we are done with estimating $h$. Similar estimate for the term $1_{x>0}e^{it\Delta}v_0$ follows easily from Strichartz-type estimates \eqref{stri1}, \eqref{stri2} and \eqref{li1}.
\end{proof}
\noindent	
In the proof of \Cref{thmm1} in the next section, we will incorporate the low frequency projection $P_{\leq 1}u$ into the spaces $X^s$ and $Y^s$, which are restricted to the time interval $T=[-1,1]$, in order to obtain the local well-posedness. Therefore, we need an estimate analogous to \eqref{main} for functions supported at low frequencies, which can be obtained from the two following propositions:
\begin{prop} Let $T=[-1,1]$. For any function $u_0\in L^2(\mathbb{R})$, we have
	\begin{equation}\label{loc}
	\abss{P_{\leq 1}e^{it\Delta}u_0}_{X_1(\mathbb{R}\times T)} \lesssim \abss{P_{\leq 1}u_0}_{L_x^2}.
	\end{equation}
\end{prop}
\begin{proof}	
	In view of Strichartz's estimate \eqref{stri1} with $p=2$ and $q=\infty$ and \eqref{li10}, it suffices to prove that
	\[	\abss{P_{\leq 1}e^{it\Delta}u_0}_{L_x^{\infty}L_t^{2}{(\mathbb{R}\times T)}} \lesssim \abss{P_{\leq 1}u_0}_{L_x^2}.\]
	Using the fact that $\widehat{P_{\leq 1}u}_0(\xi,t)$ is compactly supported in $\xi$ and Plancherel theorem, we have
	\begin{equation*}
	\begin{split}
	\abss{P_{\leq 1}e^{it\Delta}u_0}_{L^{\infty}_xL^2_t(\mathbb{R}\times T)} & \leq \abss{P_{\leq 1}e^{it\Delta}u_0}_{L^2_tL^{\infty}_x(T\times\mathbb{R})} \leq \abss{\psi(\xi) \hat{u}_0}_{L^2_tL^{1}_{\xi}(T\times\mathbb{R})} \\	&\leq  \abss{\psi(\xi) \hat{u}_0}_{L^{\infty}_tL^{2}_{\xi}(T\times\mathbb{R})} \ \ \ \ \ = \abss{P_{\leq 1}u_0}_{L_x^2}.
	\end{split}
	\end{equation*}
\end{proof}
\begin{prop}\label{propp1} Let $T=[-1,1]$. For any function $F(x,t)$ such that $P_{\leq 1}F \in Y_1$, we have
	\begin{equation}\label{estt1}
	\Big\|\int_{0}^{t}e^{i(t-s)\Delta}P_{\leq 1}F(x,s) \ ds\Big\|_{X_1(\mathbb{R}\times T)} \lesssim \abss{P_{\leq 1}F}_{Y_1(\mathbb{R}\times T)}.
	\end{equation}
\end{prop}
\begin{proof}
	As in the proof of \Cref{prop1}, it follows from Minkowski inequality that
	\[		\left\|\int_0^t e^{i(t-s)\Delta} P_{\leq 1}F(s) \ ds\right\|_{X_1(\mathbb{R}\times T)} \lesssim \abss{P_{\leq 1}F}_{L^1_tL^2_x(T\times \mathbb{R})}.
	\]
	Thus, it suffices to prove that
	\begin{equation*}
	\left\|\int_0^t e^{i(t-s)\Delta} P_{\leq 1}F(s) \ ds\right\|_{X_{1}(\mathbb{R}\times T)} \lesssim \abss{P_{\leq 1}F}_{L^1_xL^2_t(\mathbb{R}\times T)}.
	\end{equation*}
	Note that for $t\in [0,1]$, we can rewrite
	\begin{equation*}
	\begin{split}
	\int_0^t e^{i(t-s)\Delta} P_{\leq 1}F(x,s) \ ds &=\int_{\mathbb{R}}\chi_{[0,1)}(t-s)\chi_{[0,1)}(s) e^{i(t-s)\Delta} P_{\leq 1}F(x,s) \ ds \\
	&:= K(x,t)\star  \chi_{[0,1)}(t)P_{\leq 1}F(x,t)
	\end{split}
	\end{equation*}
	where $\star$ is the space-time convolution and
	\[K(x,t)=\int_{\mathbb{R}}e^{-it\xi^2+ix\xi}\chi_{[0,1)}(t)\psi\left(\frac{\xi}{N}\right) \ d\xi, \]
	which obeys the estimate \eqref{k} with $N=1$. Hence, by Young's inequality
	\[\Big\|\chi_{[0,1]}(t) \left[K(x,t)\star  \chi_{[0,1)}(t)P_{\leq 1}F(x,t)\right]\Big\|_{L_x^{2}L_t^{\infty}} \lesssim \abss{\chi_{[0,1]}(t)P_{\leq 1}F}_{L_x^{2}L_t^1}. \]
	We use the finite time restriction and apply Bernstein's and Minkowski's inequality.
	\begin{equation*}\label{loc1}
	\begin{split}
	\abss{\chi_{[0,1]}(t)P_{\leq 1}F}_{L_x^{2}L_t^1}&\lesssim\abss{\chi_{[0,1]}(t)P_{\leq 1}F}_{L_{x,t}^{2}} \\
	&\lesssim \abss{\chi_{[-1,1]}(t)P_{\leq 1}F}_{L_t^2L_x^{1}} \\
	&\leq  \abss{\chi_{[-1,1]}(t)P_{\leq 1}F}_{L_x^{1}L_t^2}.
	\end{split}
	\end{equation*}
	Since similar proof applies for the time interval $[-1,0]$, we obtain
	\[\Big\|\int_{0}^{t}e^{i(t-s)\Delta}P_{\leq 1}F(s) \ ds\Big\|_{L_x^{2}L_t^{\infty}(\mathbb{R}\times T)} \lesssim   \abss{P_{\leq 1}F}_{L_x^{1}L_t^2(\mathbb{R}\times T)}. \]
	This estimate has the following two consequences. First, from Minkowski's inequality, we have
	\begin{equation*}
	\begin{split}
	\Big\|\int_{0}^{t}e^{i(t-s)\Delta}P_{\leq 1}F(s) \ ds\Big\|_{L_t^{\infty}L_x^{2}(T\times \mathbb{R})}  &\leq 	\Big\|\int_{0}^{t}e^{i(t-s)\Delta}P_{\leq 1}F(s) \ ds\Big\|_{L_x^{2}L_t^{\infty}(\mathbb{R}\times T)} \\
	& 	\lesssim \abss{P_{\leq 1}F}_{L_x^{1}L_t^2(\mathbb{R}\times T)}.
	\end{split}
	\end{equation*}
	Secondly, it follows from Minkowski's inequality, Bernstein's inequality and the finite time restriction that
	\begin{equation*}
	\begin{split}
	\Big\|\int_{0}^{t}e^{i(t-s)\Delta}P_{\leq 1}F(s) \ ds\Big\|_{L_x^{\infty}L_t^{2}(\mathbb{R}\times T)}  &\leq 	\Big\|\int_{0}^{t}e^{i(t-s)\Delta}P_{\leq 1}F(s) \ ds\Big\|_{L_t^{2}L_x^{\infty}(T\times \mathbb{R})} \\
	& 	\lesssim \Big\|\int_{0}^{t}e^{i(t-s)\Delta}P_{\leq 1}F(s) \ ds\Big\|_{L^2_{x,t}(\mathbb{R}\times T)} \\
	& 	\lesssim \Big\|\int_{0}^{t}e^{i(t-s)\Delta}P_{\leq 1}F(s) \ ds\Big\|_{L^2_{x}L^{\infty}_t(\mathbb{R}\times T)} \\
	& 	\lesssim \abss{P_{\leq 1}F}_{L_x^{1}L_t^2(\mathbb{R}\times T)}.
	\end{split}
	\end{equation*}
	This concludes the proof of \eqref{estt1}.
\end{proof}
\noindent
The essential part of the contraction argument is a multilinear estimate: an estimate of the form $\abss{\partial_x u_1\prod_{i=2}^d u_i}_{Y^s}\lesssim \prod_{i=1}^d\abss{u_i}_{X^s}$. One of the main tools that we will use to prove this is the following Bilinear Strichartz estimate for the $X^s$ space.
\begin{theorem}\label{bithm}
	Let $N \gg M$ and suppose that $u$ and $v$ are supported at frequency $N$ and $M$, respectively. Then, we have
	\begin{equation}\label{bi2}	\abss{uv}_{L^{2}_{x,t}} \lesssim N^{-\frac{1}{2}}\abss{u}_{X_N}\abss{v}_{X_M}. \end{equation}
\end{theorem}
\begin{proof}
	Let $F_1(x,t)=(i\partial_t+\Delta)u(x,t)$ and $F_2(x,t)=(i\partial_t+\Delta)v(x,t)$. We will prove that
	\begin{align}
	\abss{uv}_{L^2_{x,t}} &\lesssim N^{-\frac{1}{2}}\left(\abss{u(0)}_{L^2_x}+\abss{F_1}_{L^1_tL^2_x}\right)\left(\abss{v(0)}_{L^2_x}+\abss{F_2}_{L^1_tL^2_x}\right) \label{part1} \\
	\abss{uv}_{L^2_{x,t}}  &\lesssim N^{-\frac{1}{2}}\left(\abss{u(0)}_{L^2_x}+N^{-\frac{1}{2}}\abss{F_1}_{L^1_xL^2_t}\right)\left(\abss{v(0)}_{L^2_x}+M^{-\frac{1}{2}}\abss{F_2}_{L^1_xL^2_t}\right) \label{part2} \\
	\abss{uv}_{L^2_{x,t}}  &\lesssim N^{-\frac{1}{2}}\left(\abss{u(0)}_{L^2_x}+N^{-\frac{1}{2}}\abss{F_1}_{L^1_xL^2_t}\right)\left(\abss{v(0)}_{L^2_x}+\abss{F_2}_{L^1_tL^2_x}\right). \label{a1} \\
	\abss{uv}_{L^2_{x,t}} &\lesssim N^{-\frac{1}{2}}\left(\abss{u(0)}_{L^2_x}+\abss{F_1}_{L^1_tL^2_x}\right)\left(\abss{v(0)}_{L^2_x}+M^{-\frac{1}{2}}\abss{F_2}_{L^1_xL^2_t}\right). \label{a2}
	\end{align}
	To achieve \eqref{part1}, we consider the expansion of $u\bar{v}$ after using the Duhamel formula on $u$ and $v$. 
	\begin{equation*}
	\begin{split}
	u(x,t)&=e^{it\Delta}u(0)-i\int_0^t e^{i(t-s)\Delta} F_1(s) \ ds \\
	v(x,t)&=e^{it\Delta}v(0)-i\int_0^t e^{i(t-s)\Delta} F_2(s) \ ds,
	\end{split}
	\end{equation*}
	It follows from the bilinear estimate for free solutions \eqref{bil} that
	\begin{align*}
	\abss{e^{it\Delta}u(0)e^{it\Delta}v(0)}_{L^2_{x,t}} &\lesssim N^{-\frac{1}{2}}\abss{u(0)}_{L^2_x}\abss{v(0)}_{L^2_x} \\
	\intertext{By the Minkowski inequality, we have that}
	\abss{\int_0^t e^{i(t-s)\Delta}F_1(s)e^{it\Delta}v(0) \ ds}_{L^2_{x,t}} &\lesssim N^{-\frac{1}{2}}\int_{\mathbb{R}} \abss{F_1(s)}_{L^2_x}\abss{v(0)}_{L^2_x} \ ds \\
	&= N^{-\frac{1}{2}} \abss{F_1}_{L^1_tL^2_x}\abss{v(0)}_{L^2_x} .
	\intertext{Similarly, }
	\abss{\int_0^t e^{it\Delta}u(0)e^{i(t-s)\Delta}F_2(s) \ ds}_{L^2_{x,t}} &\lesssim N^{-\frac{1}{2}} \abss{u(0)}_{L^2_x}\abss{F_2}_{L^1_tL^2_x} .
	\end{align*}
	With the same proof, we can estimate the last term in the product.
	\begin{equation*}
	\begin{split}
	\abss{\int_0^t \int_0^t e^{i(t-s)\Delta}F_1(s)&e^{i(t-s)\Delta}F_2(\tilde{s}) \ dsd\tilde{s}}_{L^2_{x,t}} \\
	&\lesssim	N^{-\frac{1}{2}} \abss{F_1}_{L^1_tL^2_x}\abss{F_2}_{L^1_tL^2_x} ,
	\end{split}
	\end{equation*}
	and \eqref{part1} follows. \\
	To prove \eqref{part2}, we recall \eqref{dec2} which allows us to decompose $u$ and $v$ as follows
	\begin{align}
	u(x,t)&=e^{it\Delta}u(0)-\int_{\mathbb{R}} e^{it\Delta}\mathcal{L}u_y+(P_{N/2^{50}}1_{x>0})e^{it\Delta}u_y-h_{1,y}(x,t) \ dy \label{de1} \\
	v(x,t)&=e^{it\Delta}v(0)-\int_{\mathbb{R}} e^{it\Delta}\mathcal{L}v_{y'}+(P_{M/2^{50}}1_{x>0})e^{it\Delta}v_{y'}-h_{2,y'}(x,t) \ dy' ,	\label{d2}
	\end{align}
	where $\mathcal{L}:L^2_x\to L^2_x$ is a bounded operator and $u_y,\mathcal{L}u_y$ and $h_{1,y}$ are defined similarly to \eqref{hh}, \eqref{L} and \eqref{hh}, respectively. From the remark following \eqref{dec4}, we see that these functions are supported at frequency $\sim N$. Similar $v_{y'},\mathcal{L}v_{y'},h_{2,y'}$  Moreover, we have
	\begin{equation}\label{dd22}
	\begin{split}
	\abss{\mathcal{L}u_y}_{L_x^2}+\abss{u_y}_{L_x^2}+\frac{1}{N}(\abss{\Delta h_{y}}_{L_{x,t}^2}+\abss{ \partial_t h_{y}}_{L_{x,t}^2}) &\lesssim \frac{1}{N^{\frac{1}{2}}}\abss{F_1(y,t)}_{L_t^2}.
	\end{split}
	\end{equation}
	Similar conclusions hold for $v_{y'},\mathcal{L}v_{y'}$ and $h_{2,y'}$ at frequency $\sim M$ with corresponding nonlinearity $F_2(y',t)$. Consider each term in the product $uv$. Let $\psi_{N/2^{50}}$ be the function defined by $P_{N/2^{50}}f:= \psi_{N/2^{50}}\ast f$. Observe that for any $G\in L^2$, we have that
	\begin{align*}
	\abss{(P_{N/2^{50}}&1_{x>0})e^{it\Delta}u_y G(x)}_{L^2_{x,t}} \\
	&= \abss{(\psi_{N/2^{50}}\ast 1_{x>0})e^{it\Delta}u_yG(x)}_{L^2_{x,t}} \\
	&\leq \int \abss{1_{x-z>0}e^{it\Delta}u_y(x)G(x)}_{L^2_{x,t}} \abs{\psi_{N/2^{50}}(z)} \ dz \\
	&\leq  \int \abss{e^{it\Delta}u_y(x)G(x)}_{L^2_{x,t}}\abs{ \psi_{N/2^{50}}(z)} \ dz \\
	&\lesssim \abss{e^{it\Delta}u_yG}_{L^2_{x,t}}.
	\end{align*}
	With this, we can take care of all the terms involving $P_{N/2^{50}}1_{x>0}$ in the expansion of $uv$. For any $A,B\in L^2$, we have 
	\begin{align*}
	\abss{(P_{N/2^{50}}1_{x>0})e^{it\Delta}u_y e^{it\Delta}B}_{L^2_{x,t}} &\lesssim \abss{e^{it\Delta}u_ye^{it\Delta}B}_{L^2_{x,t}} \\
	\abss{(P_{N/2^{50}}1_{x>0})e^{it\Delta}u_y h_{2,y'}}_{L^2_{x,t}} &\lesssim \abss{e^{it\Delta}u_yh_{2,y'}}_{L^2_{x,t}}. \\
	\intertext{Similarly,}
	\abss{e^{it\Delta}A(P_{N/2^{50}}1_{x>0})e^{it\Delta}v_{y'} }_{L^2_{x,t}} &\lesssim \abss{e^{it\Delta}Ae^{it\Delta}v_{y'}}_{L^2_{x,t}} \\
	\abss{h_{1,y}(P_{N/2^{50}}1_{x>0})e^{it\Delta}v_{y'} }_{L^2_{x,t}} &\lesssim \abss{h_{1,y}e^{it\Delta}v_{y'}}_{L^2_{x,t}},			
	\end{align*}
	and lastly,
	\begin{align*}
	& \ \  \ \ \Big\|\left[(P_{N/2^{50}}1_{x>0})e^{it\Delta}u_y\right] \left[(P_{N/2^{50}}1_{x>0})e^{it\Delta}v_{y'}\right] \Big\|_{L^2_{x,t}} \\
	&	\lesssim \Big\|e^{it\Delta}u_y\left[(P_{N/2^{50}}1_{x>0})e^{it\Delta}v_{y'}\right]\Big\|_{L^2_{x,t}} \\
	& \lesssim \abss{e^{it\Delta}u_ye^{it\Delta}v_{y'}}_{L^2_{x,t}}.
	\end{align*}
	Therefore, we only have to worry about the terms of the forms $e^{it\Delta}Ae^{it\Delta}B$, $e^{it\Delta}Ah_{2,y'}$, $h_{1,y}e^{it\Delta}B$ and $h_{1,y}h_{2,y'}$. Note that any choice of $A$ that is not $u(0)$ is an integral with respect to $y$. The same holds for $B$. By the bilinear Strichartz estimate \eqref{bil}, one obtains
	\begin{equation}
	\begin{split}
	\abss{e^{it\Delta}Ae^{it\Delta}B}_{L^2_{x,t}} \lesssim N^{-\frac{1}{2}}\abss{A}_{L_x^2}\abss{B}_{L_x^2}.				\end{split}
	\end{equation}
	We get the desired bound by observing that either we have $\abss{A}_{L_x^2}=\abss{u(0)}_{L_x^2} $ or $\abss{A}_{L_x^2}\lesssim\int_{\mathbb{R}}\abss{u_y}_{L_x^2} \ dy \lesssim  N^{-\frac{1}{2}}\abss{F_1}_{L_x^1L_t^2}$ from \eqref{dd22}. It remains to estimate the terms that involve $h_{1,y}$ and $h_{2,y}$. By H{\"o}lder and Bernstein inequalities, \eqref{stri2} and \eqref{h4}, We have that
	\begin{equation}\label{ext0}
	\begin{split}
	\abss{e^{it\Delta}Ah_{2,y'}}_{L^2_{x,t}} &\lesssim \abss{e^{it\Delta}A}_{L^{\infty}_{x}L^{2}_{t}}\abss{h_{2,y'}}_{L^{2}_{x}L^{\infty}_{t}} \\
	& \lesssim N^{-\frac{1}{2}}M^{-\frac{1}{2}}\abss{A}_{L_x^2}\abss{F_2(y')}_{L^2_{t}}. 
	\end{split}
	\end{equation}
	By taking $\int_{\mathbb{R}} \cdot \ dy'$ when $A=u(0)$ and $\int_{\mathbb{R}}\int_{\mathbb{R}}\cdot \ dydy'$ when $A=\mathcal{L}u_y$ or $A=u_y$ on both sides of the inequality and applying \eqref{dec1}, we get the desired bound. On the other hand, we get the estimate for $\abss{h_{1,y}e^{it\Delta}B}_{L^2_{x,t}}$ by observing that from \eqref{dec1},  we have $\abss{\Delta h_{1,y}}_{L^2_{x,t}}\lesssim N^{-\frac{3}{2}}\abss{F_1}_{L^2_t}$. Hence,
	\begin{equation}\label{ext}
	\begin{split}
	\abss{h_{1,y}e^{it\Delta}B}_{L^2_{x,t}} &\lesssim \abss{h_{1,y}}_{L^{2}_{x,t}}\abss{e^{it\Delta}B}_{L^{\infty}_{x,t}} \\
	& \lesssim N^{-\frac{3}{2}}M^{\frac{1}{2}}\abss{F_1}_{L_x^2}\abss{B}_{L^{\infty}_tL^2_{x}} \\
	&\leq  N^{-\frac{1}{2}}M^{-\frac{1}{2}}\abss{F_1}_{L_x^2}\abss{B}_{L^{\infty}_tL^2_{x}}.
	\end{split}
	\end{equation}
	\\
	Lastly, we use \eqref{h4} and \eqref{h5} to estimate the remaining term
	\begin{equation}
	\begin{split}
	\abss{h_{1,y}h_{2,y'}}_{L^2_{x,t}} &\leq \abss{h_{1,y}}_{L^{\infty}_xL^2_{t}}\abss{h_{2,y'}}_{L^{2}_x L^{\infty}_{t}} \\
	&\lesssim N^{-1}M^{-\frac{1}{2}}\abss{F_1(y)}_{L^2_{t}}\abss{F_2(y')}_{L^2_{t}}.
	\end{split}
	\end{equation}
	Taking $\int_{\mathbb{R}}\int_{\mathbb{R}}\cdot \ dydy'$, we obtain \eqref{part2}. We are now left to proving \eqref{a1} and \eqref{a2}. The proof is a mix of the ideas we used to prove \eqref{part1} and \eqref{part2}. For \eqref{a1}, we write $u$ using the decomposition \eqref{de1} and $v$ using the Duhamel formula. On the product expansion of $\abss{uv}_{L^2_{x,t}}$, we apply the triangle inequality and Minkowski inequality. We then apply the bilinear estimate \eqref{bil} to any term of the form $\abss{e^{it\Delta}Ae^{it\Delta}B}_{L^2_{x,t}}$ to get the desired bound. This leaves us with the terms of the form $\abss{e^{it\Delta}Ah_{2,y'}}_{L^2_{x,t}}$, on which we can apply \eqref{ext0}. In the same manner, we can prove \eqref{a2} using the Duhamel formula for $u$ and the decomposition \eqref{d2} for $v$. We finish the proof by observing that the terms of the form $\abss{h_{1,y}e^{it\Delta}B}_{L^2_{x,t}}$ can be bounded using \eqref{ext}. 	
\end{proof}	
\section{The Proof of \texorpdfstring{\Cref{thmm1}}{Theorem 1.1}}\label{sec4}
\noindent
Let $s$ be the exponent which satisfies the condition in \Cref{thmm1}. To obtain the local well-posedness, we redefine the spaces $X^s$ and $Y^s$ from \eqref{norm} in a way that the projections on the low frequencies are combined together. Since we assume a finite time restriction, so any spaces mentioned below are defined on the product space $\mathbb{R}\times [-1,1]$.
\begin{equation}
\begin{split}\label{norm3}
\abss{u}_{Z_N} & =\abss{u}_{L_t^{\infty}L_x^2\cap L^4_tL^{\infty}_x\cap L^6_{x,t}}+N^{-\frac{1}{2}}\abss{u}_{L_x^{2}L_t^{\infty}}+N^{\frac{1}{2}}\abss{u}_{L_x^{\infty}L_t^2} \\
\abss{u}_{Y_N} &= \inf\{N^{-\frac{1}{2}}\abss{u_1}_{L_x^1L_t^2} +\abss{u_2}_{L_{t}^1L_x^2}\ |\ u_1+u_2=u\}\\
\abss{u}_{X_N} & =\abss{u}_{Z_N} +\abss{(i\partial_t+\Delta)u}_{Y_N} \\
\abss{u}_{X^{s}} & =\abss{P_{\leq 1}u}_{X_1}+\Big(\sum_{N\in 2^{\mathbb{N}}} N^{2s}\abss{P_Nu}^2_{X_N}\Big)^{\frac{1}{2}} \\
\abss{u}_{Y^{s}}&=\abss{P_{\leq 1}u}_{Y_1}+\Big(\sum_{N\in 2^{\mathbb{N}}} N^{2s}\abss{P_Nu}^2_{Y_N}\Big)^{\frac{1}{2}}.
\end{split}
\end{equation}
\noindent The previous section prepares us all the estimates we need in order to obtain the linear estimate for the $X^s$ and $Y^s$ spaces; It follows from \eqref{main}, \eqref{loc} and \eqref{estt1} that for any $s\geq \frac{1}{2}$,
\begin{equation}\label{main1}
\abss{u}_{X^s} \lesssim \abss{u_0}_{H^s}+\abss{F}_{Y^s}.
\end{equation}
\noindent We are now ready to prove the multilinear estimate. 
\begin{theorem}\label{nonn1}  Let $d \geq 3$. For any $u_1,u_2,\ldots,u_d \in X^{s}$ where $s\geq \frac{1}{2}$, we have the following estimate. 
	\begin{equation}
	\Big\| (\partial_x u_1)\prod_{i=2}^{d}u_i\Big\|_{Y^{s}} \lesssim \prod_{i=1}^d\abss{u_i}_{X^{s}}. \label{linn6} 
	\end{equation}
\end{theorem}
\begin{proof}
	It suffices to prove that
	\begin{align}
	\Big\| (\partial_x u_1)\prod_{i=2}^{d}u_i\Big\|_{Y^{s}} \lesssim \abss{u_1}_{X^{s}}\prod_{i=2}^d\abss{u_i}_{X^{\frac{1}{2}}}. \label{linn5} 
	\end{align}
	which implies \eqref{linn6} since $X^s\subset X^{\frac{1}{2}}$ due to the absence of low frequency projections. In view of \eqref{loc} and \eqref{estt1}, we can treat $P_{\leq 1}$ as $P_1$, so it suffices to estimate the summation over high frequencies:
	\begin{equation}\label{ab}
	\begin{split}
	\sum_{N,N_1,\ldots,N_d}N^{s}\Big\|P_N(P_{N_1}\partial_x u_1\prod_{i=2}^{d}P_{N_i}u_i)\Big\|_{Y^s},
	\end{split}
	\end{equation}
	where $N\geq 1$ and $N_i\geq 1$ for all $i$ in the summation. We can assume that $N_1 \geq N_2 \geq \ldots \geq N_d$ and $N\lesssim N_1$. This is because $u_1$ is the only term in \eqref{ab} that has a derivative, and so any other frequency distribution would lead to a better estimate. We define $c_{N_1,1}=N_1^s\abss{P_{N_1}u_1}_{X_{N_1}}$ and $c_{N_i,i}=N_i^{\frac{1}{2}}\abss{P_{N_i}u_i}_{X_{N_i}}$ for $2\leq i \leq d$. Thus, we see that $\abss{c_{N_1,1}}_{l^2(N_1)}=\abss{u_1}_{X^{s}}$ and $\abss{c_{N_i,i}}_{l^2(N_i)}=\abss{u_i}_{X^{\frac{1}{2}}}$ for $2\leq i \leq d$. In order to obtain the $l^2$ summation of $c_{N_i,i}$, we will repeatedly be using the following application of the Cauchy-Schwarz inequality:
	\begin{equation}\label{tool}
	\begin{split}
	\sum_{N_j,\ldots,N_d}\frac{1}{N_j^a}\prod_{i=j}^{d}c_{N_i,i}  \leq \sum_{N_j,\ldots,N_d}\prod_{i=j}^{d}\frac{1}{N_i^{\frac{a}{d}}}c_{N_i,i} &\leq \prod_{i=j}^{d} \sum_{N_i\geq 1}\frac{1}{N_i^{\frac{a}{d}}}c_{N_i,i} \\
	&\lesssim \prod_{i=j}^{d}\abss{u_i}_{X^{\frac{1}{2}}},
	\end{split}
	\end{equation}
	for any $a>0$. To prove \eqref{linn5}, we split the summation over three different kinds of frequency interactions.
	\begin{equation*}
	\begin{split}
	&\sum_{N,N_1,\ldots,N_d}N^{s}\Big\|P_N(P_{N_1}\partial_x u_1\prod_{i=2}^{d}P_{N_i}u_i)\Big\|_{Y^s} \\&{}= \Big(\sum_{I}+\sum_{II}+\sum_{III}\Big)N^{s}\Big\|P_N(P_{N_1}\partial_x u_1\prod_{i=2}^{d}P_{N_i}u_i)\Big\|_{Y^s}.
	\end{split}
	\end{equation*}
	Each of the summations contains certain ranges of $N,N_1,\ldots, N_d$ described by the following cases:
	\begin{enumerate}[label=$\Roman*).$,wide, labelwidth=!, labelindent=0pt]
		\item $N_1 \gg N_2$ and $N\sim N_1$. \\
		By H{\"o}lder inequality, \eqref{li1} with $\gamma=2$ and \eqref{tool},
		\begin{align*}
		\sum_{N_1,\ldots,N_d}	\Big\|P_N&(P_{N_1}\partial_x u_1\prod_{i=2}^{d}P_{N_i}u_i)\Big\|_{L^1_xL^2_t} \\
		&\lesssim \sum_{N_i}\abss{P_{N_1}\partial_x u_1P_{N_2} u_2}_{L^{2}_{x,t}}\abss{P_{N_3}u_3}_{L^{2}_xL^{\infty}_t}\prod_{i=4}^{d}\abss{P_{N_i}u_i}_{L^{\infty}_{x,t}}\\
		&\lesssim \sum_{N_i}\frac{1}{N_1^{s-\frac{1}{2}}N_2^{\frac{1}{2}}}\prod_{i=1}^{d}c_{N_i,i} \\
		&\lesssim \frac{1}{N^{s-\frac{1}{2}}}\sum_{N_i}\frac{1}{N_2^{\frac{1}{2}}}\prod_{i=1}^{d}c_{N_i,i} \\
		&\lesssim \frac{1}{N^{s-\frac{1}{2}}}\sum_{N_1\sim  N}c_{N_1,1}\prod_{i=2}^{d}\abss{u_i}_{X^{\frac{1}{2}}}.
		\end{align*}
		Therefore,
		\[\sum_{I}N^{s-\frac{1}{2}}\Big\|P_N(P_{N_1}\partial_x u_1\prod_{i=2}^{d}P_{N_i}u_i)\Big\|_{L^1_xL^2_t} \lesssim \sum_{N_1\sim  N}c_{N_1,1}\prod_{i=2}^{d}\abss{u_i}_{X^{\frac{1}{2}}}.\]
		Taking the $l^2$ summation with respect to $N\geq 1$, we obtain \eqref{linn5}. \\
		\item $N_1 \sim N_2 \gg  N_3\geq \ldots \geq N_d$ and $N\lesssim N_1$.  \\
		\\
		In this case, we use the bilinear estimate for the product $P_{N_1}\partial_x u_1P_{N_3}u_3$ and put $P_{N_2}u_2$ in the Strichartz space $L^4_tL^{\infty}_x$:
		\begin{align*}
		\sum_{N_1,\ldots,N_d}	\Big\|P_N&(P_{N_1}\partial_x u_1\prod_{i=2}^{d}P_{N_i}u_i)\Big\|_{L^1_tL^2_x} \\
		&\lesssim 	\sum_{N_1,\ldots,N_d}	\Big\|P_N(P_{N_1}\partial_x u_1\prod_{i=2}^{d}P_{N_i}u_i)\Big\|_{L_t^{\frac{4}{3}}L^2_x} \\ &\lesssim \sum_{N_i}\abss{P_{N_1}\partial_x u_1P_{N_3} u_3}_{L^{2}_{t,x}}\abss{P_{N_2}u_2}_{L^{4}_tL^{\infty}_x}\prod_{i=4}^{d}\abss{P_{N_i}u_i}_{L^{\infty}_{t,x}}\\
		&\lesssim \sum_{N_i}\frac{1}{N_1^{s-\frac{1}{2}}N_2^{\frac{1}{2}}N_3^{\frac{1}{2}}}\prod_{i=1}^{d}c_{N_i,i} \\
		&\lesssim \Big(\sum_{N_1\sim N_2}\frac{1}{N_1^{s}}c_{N_1,1}c_{N_2,2}\Big)\Big(\sum_{N_3,\ldots , N_d}\frac{1}{N_3^{\frac{1}{2}}}\prod_{i=3}^{d}c_{N_i,i}\Big) \\
		&\lesssim \Big(\sum_{N_1\gtrsim N}\frac{1}{N_1^{s}}c_{N_1,1}\Big)^{\frac{1}{2}}\prod_{i=2}^{d}\abss{u_i}_{\dot{X}^{\frac{1}{2}}},
		\end{align*}
		where we used \eqref{tool} in the second to last step. Therefore,
		\[\sum_{II}	N^{2s}\abss{P_N(P_{N_1}\partial_x u_1\prod_{i=2}^{d}P_{N_i}u_i)}^2_{L^1_tL^2_x}  \lesssim  \abss{u_1}^2_{X^{s}}\prod_{i=2}^d\abss{u_i}^2_{X^{\frac{1}{2}}}.\]
		\item $N_1 \sim N_2 \sim  N_3\geq \ldots \geq N_d$ and $N\lesssim N_1$.  \\
		\\
		We divide the proof into two cases depending on the degree $d$. \\
		\begin{enumerate}[label=\textbf{\Alph*).}]
			\item $d=3$. \\
			Even though we cannot use the bilinear estimate in this case, the fact that $N_1\sim N_2 \sim N_3$ allows us to cancel the derivative loss in $P_{N_1} \partial_x u_1$ by the $\frac{1}{2}$ regularity from $P_{N_2} u_2$ and $P_{N_3} u_3$ via the H{\"o}lder inequality:
			\begin{align*}
			\sum_{N_1\sim N_2\sim N_3}&\Big\|P_N[(P_{N_1}\partial_x u_1)P_{N_2}u_2P_{N_3}u_3]\Big\|_{L^1_tL^2_x} \\&\lesssim 	\sum_{N_1\sim N_2\sim N_3}\Big\|P_N[(P_{N_1}\partial_x u_1)P_{N_2}u_2P_{N_3}u_3]\Big\|_{L^{2}_{t,x}} \\ &\lesssim 	\sum_{N_1\sim N_2\sim N_3}\abss{P_{N_1} \partial_x u_1}_{L^{6}_{t,x}}\abss{P_{N_2} u_2}_{L^{6}_{t,x}}\abss{P_{N_3} u_3}_{L^{6}_{t,x}}\\
			&\lesssim 	\sum_{N_1\sim N_2\sim N_3}\frac{N_1^{1-s}}{N_2^{\frac{1}{2}}N_3^{\frac{1}{2}}}c_{N_1,1}c_{N_2,2}c_{N_3,3} \\
			&\lesssim \Big(\sum_{N_1\gtrsim N}\frac{1}{N_1^{s}}c_{N_1,1}\Big)^{\frac{1}{2}}\abss{u_2}_{X^{\frac{1}{2}}}\abss{u_3}_{X^{\frac{1}{2}}},
			\end{align*}				
			where the last step follows from the Cauchy-Schwarz inequality on $c_{N_1,1}c_{N_2,2}c_{N_3,3}$. \\
			\item $d\geq 4$. \\
			\noindent
			We again take advantage of the finite time restriction and put $P_{N_i} u_i$ for $1\leq i \leq 4$ in suitable Strichartz spaces, namely $L^{\infty}_{t}L^{2}_x$ and $L^{4}_{t}L^{\infty}_x$.
			\begin{align*}
			\sum_{N_1,\ldots,N_d}	\Big\|P_N&(P_{N_1}\partial_x u_1\prod_{i=2}^{d}P_{N_i}u_i)\Big\|_{L^1_tL^2_x} \\
			&\lesssim 	\sum_{N_1,\ldots,N_d}	\Big\|P_N(P_{N_1}\partial_x u_1\prod_{i=2}^{d}P_{N_i}u_i)\Big\|_{L_t^{\frac{4}{3}}L^2_x} \\ &\lesssim \sum_{N_i}\abss{P_{N_1} \partial_x u_1}_{L^{\infty}_{t}L^{2}_x}\prod_{i=2}^4\abss{P_{N_i} u_i}_{L^{4}_{t}L^{\infty}_x}\prod_{i=5}^d\abss{P_{N_i} u_i}_{L^{\infty}_{t,x}}\\
			&\lesssim \sum_{N_i}\frac{N_1^{1-s}}{N_2^{\frac{1}{2}}N_3^{\frac{1}{2}}N_4^{\frac{1}{2}}}\prod_{i=1}^dc_{N_i,i} \\
			&\lesssim \Big(\sum_{N_1,N_2,N_3}\frac{1}{N_1^s}c_{N_1,1}c_{N_2,2}c_{N_3,3}\Big)\Big(\sum_{N_4\ldots , N_d}\frac{1}{N_4^{\frac{1}{2}}}\prod_{i=4}^d c_{N_i,i}\Big) \\
			&\lesssim \Big(\sum_{N_1\gtrsim N}\frac{1}{N_1^{s}}c_{N_1,1}\Big)^{\frac{1}{2}}\prod_{i=2}^{d}\abss{u_i}_{X^{\frac{1}{2}}},
			\end{align*}	
		\end{enumerate}
		In either case, it follows that
		\[\sum_{III} N^{2s}\Big\|P_N(P_{N_1}\partial_x u_1\prod_{i=2}^{d}P_{N_i}u_i)\Big\|^2_{L^1_tL^2_x} \lesssim \abss{u_1}^2_{X^{s}}\prod_{i=2}^d\abss{u_i}^2_{X^{\frac{1}{2}}}.\]
		and this concludes the proof.
	\end{enumerate}
\end{proof}
\noindent
In view of this theorem, if every term in $P(u,\bar{u},\partial_x u,\partial_x \bar{ u})$ has only one derivative, then we expect to close the contraction argument in a subspace of $X^{\frac{1}{2}}$. On the other hand, if we replace $u_j$ by $\partial_x u_j$ for some $j\geq 2$, then it follows from \eqref{p1} that $\abss{\partial_x u_i}_{X^s}\lesssim \abss{u_i}_{X^{s+1}}$ for any $s>0$, and so \eqref{linn5} yields
\begin{align*}
\Big\| (\partial_x u_1)(\partial_x u_j)\prod_{\substack{i=2 \\ i\not= j}}^{d}u_i\Big\|_{Y^{\frac{3}{2}}} 
&\lesssim \abss{u_1}_{X^{\frac{3}{2}}}\abss{\partial_x u_j}_{X^{\frac{1}{2}}}\prod_{\substack{i=2 \\ i\not= j}}^d\abss{u_i}_{X^{\frac{1}{2}}} \\
&\lesssim \abss{u_1}_{X^{\frac{3}{2}}}\prod_{i=2}^d\abss{u_i}_{X^{\frac{3}{2}}}, 
\shortintertext{and for any $s\geq \frac{3}{2}$, we have}
\Big\|  (\partial_x u_1)(\partial_x u_j)\prod_{\substack{i=2 \\ i\not= j}}^{d}u_i\Big\|_{Y^s} 
&\lesssim \abss{u_1}_{X^s}\prod_{i=2}^d\abss{u_i}_{X^s}. 
\end{align*}
Consequently, in the case that a term in $P(u,\bar{u},\partial_x u,\partial_x \bar{u})$ has more than one derivative, we can employ the contraction argument in $X^{\frac{3}{2}}$.
\begin{proof}[Proof of \Cref{thmm1}]
	We define $F(u):=P(u,\bar{u},\partial_x u,\partial_x \bar{ u})$. Let $u$ and $v$ be functions in $X^s$. We use the main linear estimate \eqref{main1} and simple algebra to obtain
    \begin{equation}\begin{split}\label{proof}  \Big\| \int_0^t  & e^{i(t-s)\partial_x^2}[F(u(x,s))-F(v(x,s))] \ ds  \Big\|_{X^s} \\ 
	&\leq c_1\left\|F(u)-F(v) \right\|_{Y^s} \\
	&\leq c_1c_2(\abss{u}^{d-1}_{X^s}+\abss{v}^{d-1}_{X^s})\abss{u-v}_{X^s},\end{split}\end{equation}
	where we used the multilinear estimate \eqref{linn6} in the last step. \\
	
	\noindent Let $C:=\min\left\{(8c_1c_2)^{-\frac{1}{d-1}},(4c_2)^{-\frac{1}{d-1}}\right\}$ where $c_1$ and $c_2$ are constants in \eqref{proof}. Define a Banach space as stated in the theorem: 
    \[X=\{u\in C_t^0H_x^s([-1,1]\times\mathbb{R}) \cap X^s : \abss{u}_{X^s}\leq 2C\}.\]  
    Let $u_0\in X$ such that $\abss{u_0}_{H^s}\leq C$. Then, for $u \in X$, we define an operator
	\[Lu := e^{it\Delta}u_0-i\int_{0}^{t}e^{i(t-s)\Delta} F(u(x,s)) \ ds,\]
	Again, by the main linear estimate, we have
	\begin{equation*}\begin{split}\abss{Lu}_{X^s}&\leq \abss{u_0}_{H^s}+\left\|F \right\|_{Y^s} \\ &\leq \abss{u_0}_{H^s}+c_2\abss{u}^d_{X^s} \\ & \leq \frac{3C}{2}<2C.\end{split}\end{equation*}
	Thus, $L$ maps $X$ to $X$.
	Moreover, from \eqref{proof},
	\[\abss{Lu-Lv}_{X^s} \leq c_1c_2(\abss{u}^{d-1}_{X^s}+\abss{v}^{d-1}_{X^s})\abss{u-v}_{X^s} \leq \frac{1}{4}\abss{u-v}_{X^s}.\]
	Thus, $L$ is a contraction and the local well-posedness in $X$ immediately follows. 
\end{proof}

\section{The Proof of \texorpdfstring{\Cref{thmm}}{Theorem 1.3} when \texorpdfstring{$d\geq 6$}{d>=6}}\label{sec6}
\noindent
In the previous sections, we used the time restriction to avoid dealing with low frequencies at $\xi\leq 1$. However, such argument cannot be used to obtain the global well-posedness for the gDNLS with nonlinearity of order $d\geq 5$. Therefore, the function spaces that we use will take these low frequencies into account. Let $s_0(d)=\frac{1}{2}-\frac{1}{d-1}=\frac{d-3}{2(d-1)}$ for $d\geq 5$. The spaces $X^s$ and $Y^s$ in \eqref{norm} are replaced by those defined by the quasi-norms $\dot{X}^s$ and $\dot{Y}^s$ which in turn are defined by the norms $X_N$ and $Y_N$,
\begin{align*}
\abss{u}_{X_N} & =\abss{u}_{L_t^{\infty}L_x^2}+N^{-\frac{1}{4}}\abss{u}_{L_x^{4}L_t^{\infty}}+N^{\frac{1}{2}}\abss{u}_{L_x^{\infty}L_t^2} \\
& \ \ \ +N^{-\frac{1}{2}}\abss{(i\partial_t+\Delta)u}_{L_x^1L_t^2} \\
\abss{u}_{\dot{X}^{s}} & =\Big(\sum_{N\in 2^{\mathbb{Z}}} N^{2s}\abss{P_Nu}^2_{X_N}\Big)^{\frac{1}{2}}  \\
\abss{u}_{X^s} &= \abss{u}_{\dot{X}^0}+\abss{u}_{\dot{X}^s} \numberthis \label{norm2}\\
\abss{u}_{Y_N} &= N^{-\frac{1}{2}}\abss{u}_{L_x^1L_t^2} \\
\abss{u}_{\dot{Y}^{s}}&=\Big(\sum_{N\in 2^{\mathbb{Z}}} N^{2s}\abss{P_Nu}^2_{Y_N}\Big)^{\frac{1}{2}} \\
\abss{u}_{Y^s} &= \abss{u}_{\dot{Y}^0}+\abss{u}_{\dot{Y}^s} .
\end{align*}
\noindent
Thus we have embeddings $X^s\hookrightarrow H^s$ and $X^{s}\hookrightarrow X^{s_0}\hookrightarrow \dot{X}^{s_0}$ for any $s\geq s_0$. In view of \eqref{main}, we obtain the linear estimate
\begin{equation}\label{main2}
\abss{u}_{X^s} \lesssim \abss{u_0}_{H^s}+\abss{F}_{Y^s}.
\end{equation}
With these choices of spaces, we can establish the multilinear estimate for $d\geq 6$. The proof for the case $d=5$ is significantly more involved and requires some frequency-modulation analysis, so we will postpone it to the next section.
\begin{theorem} \label{thmnon2} Let $d\geq 6$. We have the following estimates. \\ 
	1). For any $u_1,u_2,\ldots,u_d \in X^{s_0}$,
	\begin{align}
	\Big\| \partial_x\prod_{i=1}^{d}u_i\Big\|_{\dot{Y}^{s_0}} &\lesssim \prod_{i=1}^d\abss{u_i}_{\dot{X}^{s_0}}, \label{linn1} 
	\end{align}
	2). Let $s\geq s_0$. For any $u_1,u_2,\ldots,u_d \in X^{s}$,
	\begin{align}
	\Big\| \partial_x\prod_{i=1}^{d}u_i\Big\|_{Y^s} &\lesssim \prod_{i=1}^d\abss{u_i}_{X^s}. \label{linn2}
	\end{align}
\end{theorem}
\begin{proof}
	Our goal is to obtain the estimate
	\begin{equation}\label{goal}
	\sum_N N^{2s+1}\abss{P_N\prod_{i=1}^{d}u_i}^2_{L_x^1L_t^2} \lesssim  \sum_{j=1}^d\abss{u_j}^2_{\dot{X}^{s}}\prod_{i\not= j}\abss{u_i}^2_{\dot{X}^{s_0}},
	\end{equation}
	which implies \eqref{linn1} by choosing $s=s_0$. We get \eqref{linn2} by combining two different versions of this estimate with a fixed $s\geq s_0$ and with $s=0$. We will focus on each term on the left-hand side of \eqref{linn2}
	\begin{equation*}
	\begin{split}
	N^{2s-1}\Big\|P_N\partial_x\prod_{i=1}^d u_i\Big\|^2_{L^1_xL^2_t} &= N^{2s-1}\Big\|P_N\partial_x\sum_{N_1,\ldots,N_d}\prod_{i=1}^{d}P_{N_i}u_i\Big\|^2_{L^1_xL^2_t} \\
	&\lesssim N^{2s+1}\sum_{N_1,\ldots,N_d}\Big\|P_N\prod_{i=1}^{d}P_{N_i}u_i\Big\|^2_{L^1_xL^2_t} ,
	\end{split}
	\end{equation*}
	and study different kinds of frequency interactions. 
	As before, we assume that $N_1 \geq N_2 \geq \ldots \geq N_d$. We define $c_{N_1,1}=N_1^s\abss{P_{N_i}u_1}_{X_{N_1}}$ and $c_{N_i,i}=N_i^{s_0}\abss{P_{N_i}u_i}_{X_{N_i}}$ for $2\leq i \leq d$. We will use the following two estimates for a product of terms with higher and lower frequencies.\\
	\begin{enumerate}
		\item For $N\lesssim N_1\sim N_2 \sim \ldots \sim N_{j-1}$ where $j\geq 3$, it follows from the Cauchy-Schwarz inequality that
		\begin{equation}\label{cs}
		\sum_{N_i}\prod_{i=1}^{j-1} c_{N_i,i}  \lesssim  \Big(\sum_{N_1 \gtrsim N}c^2_{N_1,1}\Big)^{\frac{1}{2}}\prod_{i=2}^{j-1}\abss{u_i}_{\dot{X}^{s_0}}.
		\end{equation}
		\item For $N_{j}\geq N_{j+1}\geq \ldots \geq N_d$ and any $\alpha >0$, Young's inequality and trivial estimate $c_{N_i,i}\leq \abss{u_i}_{\dot{X}^{s_0}}$ imply
		\begin{equation}\label{yi}
		\begin{split}
		\sum_{N_{j}\geq \ldots \geq N_d}\left(\frac{N_d}{N_{j}}\right)^{\alpha}\prod_{i=j}^{d} c_{N_i,i}  &\leq \prod_{i=j+1}^{d-1} \abss{u_i}_{\dot{X}^{s_0}}\sum_{N_{j}\geq \ldots \geq N_d} \left(\frac{N_d}{N_{j}}\right)^{\alpha}c_{N_j,j}c_{N_d,d} \\
		&\lesssim_{\alpha} \prod_{i=j}^{d} \abss{u_i}_{\dot{X}^{s_0}}.
		\end{split}
		\end{equation}	
	\end{enumerate}
	These estimates will be used in each case after appropriate uses of H{\"o}lder inequality, Bernstein inequality and bilinear estimate \eqref{bi2}. \\
	By H{\"o}lder and Bernstein inequalities,
	\begin{equation*}
	\begin{split}
	\Big\|P_N&\prod_{i=1}^{d}u_i\Big\|_{L^1_xL^2_t} \\ 
	&\lesssim \sum_{N_i}\abss{P_{N_1}u_1}_{L^{\infty}_{x}L^{2}_{t}}\prod_{i=2}^{5}\abss{P_{N_i}u_i}_{L^{4}_xL^{\infty}_t}\prod_{i=6}^{d}\abss{P_{N_i}u_i}_{L^{\infty}_{x,t}} \\
	&\lesssim \sum_{N_i}\abss{P_{N_1}u_1}_{L^{\infty}_{x}L^{2}_{t}}\prod_{i=2}^{5}\abss{P_{N_i}u_i}_{L^{4}_xL^{\infty}_t}\prod_{i=6}^{d}N_i^{\frac{1}{2}}\abss{P_{N_i}u_i}_{L^{\infty}_{t}L^{2}_{x}} \\
	&\lesssim \sum_{N_i}\frac{1}{N_1^{s+\frac{1}{2}}}\prod_{i=2}^{5}\frac{1}{N_i^{{s_0}-\frac{1}{4}}}\prod_{i=6}^{d}N_i^{\frac{1}{2}-s_0}\prod_{i=1}^{d}c_{N_i,i}.
	\end{split}
	\end{equation*}
	Since $s_0=\frac{1}{2}-\frac{1}{d-1}$, the sums of the exponents in $\prod_{i=2}^{5}N_i^{{s_0}-\frac{1}{4}}$ and $\prod_{i=6}^{d}N_i^{\frac{1}{2}-s_0}$ are equal. With the assumption that $N_2\geq N_3\geq \ldots \geq N_d$, the right-hand side is bounded by
	\begin{equation}\label{d6}
	\sum_{N_i}\frac{1}{N_1^{s+\frac{1}{2}}}\left(\frac{N_d}{N_2}\right)^{\frac{1}{4(d-1)}} \prod_{i=1}^{d}c_{N_i,i}.
	\end{equation}
	To estimate this term, we consider the following two frequency interactions. \\
	\begin{enumerate}
		\item $N\sim N_1 \gg N_2\geq \ldots \geq N_d$. \\
		Using \eqref{yi} on $c_{N_2,2}c_{N_3,3}\ldots c_{N_d,d}$, we can bound \eqref{d6} by
		\[\sum_{N_1\sim N}\frac{1}{N_1^{s+\frac{1}{2}}}c_{N_1,1}\prod_{i=2}^{5}\abss{u_i}_{\dot{X}^{s_0}},\]
		for each fixed $N$. We have that
		\begin{equation*}
		\begin{split}
		\sum_N \Big(\sum_{N_1\sim N}\frac{N^{2s+1}}{N_1^{s+\frac{1}{2}}}c_{N_1,1}\Big)^2\prod_{i=2}^{5}\abss{u_i}^2_{\dot{X}^{s_0}} &\sim  \sum_{N_1}c_{N_1,1}^2\prod_{i=2}^{5}\abss{u_i}^2_{\dot{X}^{s_0}} \\
		&= \abss{u_1}^2_{\dot{X}^s}\prod_{i=2}^{5}\abss{u_i}^2_{\dot{X}^{s_0}} ,
		\end{split}
		\end{equation*}
		which implies \eqref{goal} as desired. \\
		\item $N\lesssim N_1 \sim N_2\geq \ldots \geq N_d$. \\
		Using \eqref{cs} on $c_{N_1,1}c_{N_2,2}$ and \eqref{yi} on $c_{N_3,3}c_{N_4,4}\ldots c_{N_d,d}$, we can bound \eqref{d6} by
		\[\Big(\sum_{N_1\gtrsim N}\frac{1}{N_1^{2s+1}}c^2_{N_1,1}\Big)^{\frac{1}{2}}\prod_{i=2}^{5}\abss{u_i}_{\dot{X}^{s_0}}.\]
		Therefore, by switching the order of summations,
		\begin{equation*}
		\begin{split}
		\sum_N \sum_{N_1\gtrsim N}\frac{N^{2s+1}}{N_1^{2s+1}}c^2_{N_1,1}\prod_{i=2}^{5}\abss{u_i}^2_{\dot{X}^{s_0}} &\lesssim  \sum_{N_1}c_{N_1,1}^2\prod_{i=2}^{5}\abss{u_i}^2_{\dot{X}^{s_0}},
		\end{split}
		\end{equation*}
		which again implies \eqref{goal}. This concludes the proof for $d\geq 6$.
	\end{enumerate}
\end{proof}
\noindent
Using the linear estimate \eqref{main2} and the multilinear estimates \eqref{linn1} and \eqref{linn2}, the proof for \Cref{thmm} follows in the same manner as in \Cref{thmm1}. Note that we did not use any finite time restriction in any parts of the proof.
\bigskip
\section{The Proof of \texorpdfstring{\Cref{thmm}}{Theorem 1.3} when \texorpdfstring{$d=5$}{d=5}}\label{sec7}
\noindent
The difficulty in this case arises from the fact that there is no room left to put the lowest frequency term in $L^{\infty}_{x,t}$. Thus, we will take this case with extra care by adding the $\dot{X}^{0,b,q}$ spaces. For each $N\in 2^{\mathbb{Z}}$, let $A_N$ be a set defined by
\begin{equation}\label{set3}A_M:=\{(\xi,\tau):M\leq \abs{\tau +\xi^2} \leq 2M\}.\end{equation}
Recall that $\tilde{u}(\xi,\tau)$ is the space-time Fourier transform of $u(x,t)$. The $\dot{X}^{0,b,q}$ space is the closure of the test functions under the following norm:
\begin{equation*}
\abss{u}_{\dot{X}^{0,b,q}} :=\Big(\sum_{M\in 2^{\mathbb{Z}}}(N^b\abss{\tilde{u}}_{L^2_{\xi,\tau}(A_M)})^q\Big)^{\frac{1}{q}}.
\end{equation*} 
Previously, the nonlinear space $\dot{Y}^s$ is based on the space $Z_N$ defined by the following norm on each frequency $N$.
\[\abss{u}_{Z_N}:= N^{-\frac{1}{2}}\abss{u}_{L_x^1L_t^2}.\] 
We modify this by adding the $\dot{X}^{0,-\frac{1}{2},1}$ space.
\[Y_N :=  Z_N +  \dot{X}^{0,-\frac{1}{2},1}.\]
The solution space is defined by
\begin{align*}
\abss{u}_{X_N} &= \abss{u(0)}_{L^2_x} +\abss{(i\partial_t+\Delta)u}_{Y_N} \\
\abss{u}_{\dot{X}^{s}} & =\Big(\sum_{N\in 2^{\mathbb{Z}}} N^{2s}\abss{P_Nu}^2_{X_N}\Big)^{\frac{1}{2}} \\
\abss{u}_{X^s} &= \abss{u}_{\dot{X}^0}+\abss{u}_{\dot{X}^s}, \numberthis \label{xsd5}
\end{align*} 
and the nonlinear space is defined by
\begin{equation}
\begin{split}
\abss{u}_{\dot{Y}^{s}}&=\Big(\sum_{N\in 2^{\mathbb{Z}}} N^{2s}\abss{P_Nu}^2_{Y_N}\Big)^{\frac{1}{2}} \\
\abss{u}_{Y^s} &= \abss{u}_{\dot{Y}^0}+\abss{u}_{\dot{Y}^s}. 
\end{split}
\end{equation}
The following proposition shows that any estimates of free solutions that we proved in \Cref{sec2} can be extended to functions in $X_N$ using the Schr{\"o}dinger equation version of Lemma 4.1 from Tao (\cite{MR2286393}). 
\begin{prop}[\cite{MR2286393}]\label{prop5}
	Let $S$ be any space-time Banach space that satisfies the following inequality,
	\begin{equation}\label{inf1}
	\abss{g(t)F(x,t)}_S \ \leq \abss{g}_{L^{\infty}_t}\abss{F(x,t)}_S,
	\end{equation}
	for any $F\in S$ and $g\in L^{\infty}_t(\mathbb{R})$. Let $T:L^2(\mathbb{R})\times\ldots \times L^2(\mathbb{R}) \to S$ be a spatial multilinear operator satisfying
	\begin{equation*}
	\abss{T(e^{it\Delta}u_{1,0},\ldots , e^{it\Delta}u_{k,0})}_S \ \lesssim \prod_{i=1}^k \abss{u_{i,0}}_{L^2_x}
	\end{equation*}
	for any $u_{1,0},\ldots , u_{k,0}\in L^2_x(\mathbb{R})$. Then the following estimate 
	\begin{equation}\label{TT}
	\abss{T(u_1,\ldots , u_k)}_S \ \lesssim \prod_{i=1}^k (\abss{u_i(0)}_{L^2_x}+\abss{(i\partial_t+\Delta)u_i}_{ \dot{X}^{0,-\frac{1}{2},1}})
	\end{equation}
	holds true for any $u_1,\ldots , u_k \in  \dot{X}^{0,-\frac{1}{2},1}$ provided that $u_i$ is supported at frequency $\sim N_i$ for $1\leq i \leq k$.
\end{prop}
\noindent
With this proposition, we can obtain several Strichartz-type estimates for $X_N$ that will be useful later on.
\begin{corollary}For any $u\in X_N$, we have the following estimates:
	\begin{align}
	\abss{u}_{L^{\infty}_tL^2_x\cap L^6_{t,x}} &\lesssim \abss{u}_{X_N} \label{striX} \\
	\abss{u}_{L^{\infty}_xL^2_t} &\lesssim N^{-\frac{1}{2}}\abss{u}_{X_N} \\
	\abss{u}_{L^4_xL^{\infty}_t} &\lesssim N^{\frac{1}{4}}\abss{u}_{X_N} 
	\end{align}
\end{corollary}
\begin{proof}
	We apply \Cref{prop5} to linear estimates \eqref{stri1}, \eqref{stri2} and \eqref{li1}, and bilinear estimates \eqref{bi15} and \eqref{bil}. 
\end{proof}
\noindent We also have the bilinear estimate adapted to the space $X_N$.
\begin{prop}\label{propp2}
	Let $N,M$ and $\lambda$ be dyadic numbers such that $M\leq N$ and $\lambda \lesssim N$. For any functions $u$ and $v$ supported at frequency $\sim N$ and $\sim M$, respectively, we have
	\begin{align}
	\abss{P_{>\lambda}(u\bar{v})}_{L^{2}_{t,x}} &\lesssim \lambda^{-\frac{1}{2}}\abss{u}_{X_N}\abss{v}_{X_M} \label{bi25} \\
	\intertext{In addition, if $\hat{u}$ and $\hat{v}$ have disjoint supports and $\alpha=\inf\abs{ \supp(\hat{u})-\supp(\hat{v})}$, then we have} 
	\abss{uv}_{L^{2}_{t,x}} &\lesssim \alpha^{-\frac{1}{2}}\abss{u}_{X_N}\abss{v}_{X_M}. \label{bi3} 
	\end{align}
\end{prop}
\begin{proof}
	As before, the bilinear estimate for homogeneous solutions \eqref{bi15} and \eqref{bil} is the keys to proving these estimates. It suffices to prove \eqref{bi25}, since \eqref{bi3} will follow in a similar manner. Denote $F_1:=(i\partial_t+\Delta)u$ and $F_2:=(i\partial_t+\Delta)v$. Using \Cref{prop5} with $T(u_1,u_2)=u_1u_2$ to extend the bilinear estimate \eqref{bi15}, we obtain
	\[\abss{P_{>\lambda}(u\bar{v})}_{L^{2}_{t,x}} \lesssim \lambda^{-\frac{1}{2}}(\abss{u(0)}_{L^2_x}+\abss{F_1}_{\dot{X}^{0,-\frac{1}{2},1}})(\abss{v(0)}_{L^2_x}+\abss{F_2}_{\dot{X}^{0,-\frac{1}{2},1}}).\]
	Therefore, it suffices to prove that for any $u\in X_N$ and $v\in X_M$,
	\begin{align}
	\label{10}\abss{P_{>\lambda}(u\bar{v})}_{L^{2}_{t,x}} &\lesssim \lambda^{-\frac{1}{2}}(\abss{u(0)}_{L^2_x}+\abss{F_1}_{Z_N})(\abss{v(0)}_{L^2_x}+\abss{F_2}_{Z_M}), \\
	\label{11}
	\abss{P_{>\lambda}(u\bar{v})}_{L^{2}_{t,x}} &\lesssim \lambda^{-\frac{1}{2}}(\abss{u(0)}_{L^2_x}+\abss{F_1}_{Z_N})(\abss{v(0)}_{L^2_x}+\abss{F_2}_{\dot{X}^{0,-\frac{1}{2},1}}), \\
	\label{12}\abss{P_{>\lambda}(u\bar{v})}_{L^{2}_{t,x}} &\lesssim \lambda^{-\frac{1}{2}}(\abss{u(0)}_{L^2_x}+\abss{F_1}_{\dot{X}^{0,-\frac{1}{2},1}})(\abss{v(0)}_{L^2_x}+\abss{F_2}_{Z_N}).
	\end{align}
	We use the decomposition from \eqref{dec2} for $u$. However, in this case, the frequency localization at $\frac{N}{2^{50}}$ is replaced by $\frac{\lambda}{2^{50}}$:
	\begin{equation*}
	u(x,t)=e^{it\Delta}u(0)-\int_{\mathbb{R}} e^{it\Delta}\mathcal{L}u_y+(P_{\lambda/2^{50}}1_{x>0})e^{it\Delta}u_y-h_{y}(x,t) \ dy,	
	\end{equation*}
	where $\mathcal{L}:L^2_x\to L^2_x$ is a bounded operator and $u_y,\mathcal{L}u_y$ and $h_y$ are defined similarly to \eqref{v0}, \eqref{L} and \eqref{hh}, respectively. From the remark following \eqref{dec4}, we see that these functions are supported at frequency $\sim N$. Moreover, the following estimate still holds even with the frequency replacement.
	\begin{equation}\label{dd2}
	\begin{split}
	\abss{\mathcal{L}u_y}_{L_x^2}+\abss{u_y}_{L_x^2}+\frac{1}{N}(\abss{\Delta h_{y}}_{L_{x,t}^2}+\abss{ \partial_t h_{y}}_{L_{x,t}^2}) &\lesssim \frac{1}{N^{\frac{1}{2}}}\abss{F_1(y,t)}_{L_t^2}.
	\end{split}
	\end{equation}
	We consider all the possible terms in $P_{>\lambda}(u\bar{v})$. First, we consider all the terms that involve $P_{\lambda/2^{50}}1_{x>0}$. For any $G\in L^2_x$, we have that
	\begin{align*}
	P_{>\lambda}\Big[(P_{\lambda/2^{50}}1_{x>0})e^{it\Delta}u_y G\Big] = &P_{>\lambda}\Big[(P_{\lambda/2^{50}}1_{x>0})P_{\ll\lambda}(e^{it\Delta}u_y G)\Big] \\
	& {} + P_{>\lambda}\Big[(P_{\lambda/2^{50}}1_{x>0})P_{\gtrsim\lambda}(e^{it\Delta}u_y G)\Big] \\
	= & P_{>\lambda}\Big[(P_{\lambda/2^{50}}1_{x>0})P_{\gtrsim\lambda}(e^{it\Delta}u_y G)\Big].
	\end{align*}
	Let $\psi_{N/2^{50}}$ be the function defined by $P_{N/2^{50}}f:= \psi_{N/2^{50}}\ast f$. Consequently,
	\begin{align*}
	\Big\|P_{>\lambda}&\left[(P_{\lambda/2^{50}}1_{x>0})e^{it\Delta}u_y G\right]\Big\|_{L^2_{x,t}} \\
	&= \Big\| P_{>\lambda}\left[(P_{\lambda/2^{50}}1_{x>0})P_{\gtrsim\lambda}(e^{it\Delta}u_y G)\right]\Big\|_{L^2_{x,t}} \\
	&\lesssim \Big\| (P_{\lambda/2^{50}}1_{x>0})P_{\gtrsim\lambda}(e^{it\Delta}u_y G)\Big\|_{L^2_{x,t}} \\
	&= \Big\|(\psi_{\lambda/2^{50}}\ast 1_{x>0})P_{\gtrsim\lambda}(e^{it\Delta}u_y G)\Big\|_{L^2_{x,t}} \\
	&\leq \int \Big\| 1_{x-z>0}P_{\gtrsim\lambda}\left[e^{it\Delta}u_y(x)G(x)\right]\Big\|_{L^2_{x,t}} \abs{\psi_{N/2^{50}}(z)} \ dz \\
	&\lesssim   \abss{P_{\gtrsim\lambda}(e^{it\Delta}u_yG)}_{L^2_{x,t}}.
	\end{align*}
	In other words, to estimate such terms, we can take out the $P_{\lambda/2^{50}}1_{x>0}$ factor just like what we did in the proof of \Cref{bithm}. Following the same line of proof as for \eqref{part2} but using a different bilinear estimate \eqref{bi15} instead of \eqref{bil}, we obtain \eqref{10}. To prove \eqref{11} and \eqref{12}, we will show that for any $v_0\in L^2_x$ supported at frequency $\sim M$,
	\begin{equation}\label{22}
	\abss{P_{>\lambda}(u\overline{e^{it\Delta}v_0})}_{L^2_{x,t}} \lesssim \lambda^{-\frac{1}{2}}(\abss{u(0)}_{L^2_x}+\abss{F_1}_{Z_N})\abss{v_0}_{L^2_x},
	\end{equation}
	which, in view of \Cref{prop5} with $T(v)=P_{>\lambda}(u\bar{v})$, leads to \eqref{11}. 
	From \eqref{bi15} and \eqref{dd2}, we obtain
	\begin{align*}
	\abss{P_{>\lambda}(e^{it\Delta}u(0)\overline{e^{it\Delta}v_0})}_{L^2_{x,t}} &\lesssim  \lambda^{-\frac{1}{2}}\abss{u(0)}_{L^2_x}\abss{v_0}_{L^2_x} ,\\
	\abss{P_{>\lambda}(e^{it\Delta}\mathcal{L}u_y\overline{e^{it\Delta}v_0})}_{L^2_{x,t}} &\lesssim  \lambda^{-\frac{1}{2}}\abss{\mathcal{L}u_y}_{L^2_x}\abss{v_0}_{L^2_x},\\ 
	&\lesssim (\lambda N)^{-\frac{1}{2}}\abss{F_1(y,t)}_{L_t^2}\abss{v_0}_{L^2_x} \\
	\abss{P_{\gtrsim\lambda}(e^{it\Delta}u_y\overline{e^{it\Delta}v_0})}_{L^2_{x,t}}  &\lesssim \lambda^{-\frac{1}{2}}\abss{u_y}_{L^2_x}\abss{v_0}_{L^2_x}\\
	&\lesssim (\lambda N)^{-\frac{1}{2}}\abss{F_1(y,t)}_{L_t^2}\abss{v_0}_{L^2_x} .
	\end{align*}
	We use the last inequality to estimate the term in $P_{>\lambda}(u\overline{e^{it\Delta}v_0})$ that involves $P_{\lambda/2^{50}}1_{x>0}$.
	\begin{equation*}
	\begin{split}
	\Big\|P_{>\lambda}\left[(P_{\lambda/2^{50}}1_{x>0})e^{it\Delta}u_y \overline{e^{it\Delta}v_0}\right]\Big\|_{L^2_{x,t}} &\lesssim   \abss{P_{\gtrsim\lambda}(e^{it\Delta}u_y\overline{e^{it\Delta}v_0})}_{L^2_{x,t}} \\
	& \lesssim (\lambda N)^{-\frac{1}{2}}\abss{F_1(y,t)}_{L_t^2}\abss{v_0}_{L^2_x}.
	\end{split}
	\end{equation*}
	For the remaining term, we use the H{\"o}lder inequality, \eqref{dd2} and the fact that $\lambda \lesssim N$.
	\begin{equation}\label{13}
	\begin{split}
	\abss{P_{>\lambda}(h_y\overline{e^{it\Delta}v_0})}_{L^2_{x,t}} &\lesssim \abss{h_y}_{L^2_{x,t}}\abss{e^{it\Delta}v_0}_{L^{\infty}_{x,t}} \\
	&\lesssim \frac{M^{\frac{1}{2}}}{N^{\frac{3}{2}}}\abss{F_1(y,t)}_{L_t^2}\abss{v_0}_{L^2_x} \\
	&\lesssim (\lambda N)^{-\frac{1}{2}}\abss{F_1(y,t)}_{L_t^2}\abss{v_0}_{L^2_x} .
	\end{split}
	\end{equation}
	Recalling that $\abss{(i\partial_t+\Delta)u}_{Z_N}=N^{-\frac{1}{2}}\abss{(i\partial_t+\Delta)u}_{L^1_xL^2_t}$, these estimates yield \eqref{22} via the Minkowski inequality. The proof for \eqref{12} is similar, except at \eqref{13} where we have the following modification:
	\begin{equation*}
	\begin{split}
	\abss{P_{>\lambda}(e^{it\Delta}u_0\overline{h}_{y'})}_{L^2_{x,t}} &\lesssim \abss{e^{it\Delta}u_0}_{L^{\infty}_xL^2_{t}}\abss{h_{y'}}_{L^2_xL^{\infty}_{t}} \\
	&\lesssim (NM)^{-\frac{1}{2}}\abss{u_0}_{L^2_x}\abss{F_2(y',t)}_{L_t^2}\\
	&\lesssim (\lambda M)^{-\frac{1}{2}}\abss{u_0}_{L^2_x}\abss{F_2(y',t)}_{L_t^2} .
	\end{split}
	\end{equation*}
	For the second to last inequality, we used the smoothing estimate \eqref{stri2} and \eqref{h2} with $d=3$. This concludes the proof of \eqref{bi25}.
\end{proof}
\noindent We will also use the following estimate which was taken from Tao (\cite{MR2286393}) and modified to be suitable to our spaces.
\begin{prop}
	Suppose that $u$ is supported at frequency $\sim N$. Then we have
	\begin{equation}\label{xsb}
	\abss{u}_{\dot{X}^{0,\frac{1}{2},\infty}} \lesssim \abss{u}_{X_N}.
	\end{equation}
\end{prop}
\begin{proof}
	Consider the Duhamel's formula of $u$. 
	\begin{equation}u(x,t)=e^{it\Delta}u_0-i\int_0^t e^{i(t-s)\Delta} F_1(s) \ ds-i\int_0^t e^{i(t-s)\Delta} F_2(s) \ ds,\end{equation}
	where $F_1\in Z_N$ and $F_2\in \dot{X}^{0,-\frac{1}{2},1}$. For $i=1,2$, we split the term
	\[\int_0^t e^{i(t-s)\Delta} F_i(s) \ ds=\int_{-\infty}^t e^{i(t-s)\Delta} F_i(s) \ ds-e^{it\Delta}\int_{-\infty}^0 e^{-is\Delta} F_i(s) \ ds.\]
	Since the $\dot{X}^{0,\frac{1}{2},\infty}$ seminorm vanishes on any free solution, it suffices to estimate the first term. For $F_1$, we recall the computation \eqref{inhom} from the proof of \Cref{lem1} that the first term is equal to
	\[\int w_y \ dy \ \ \ \text{ where } \ \ \  \widetilde{w}_y = \frac{\psi_N(\xi)}{-\tau-\xi^2-i0}\widehat{F}_1(y,\tau).\]
	With a direct integration, we see that
	\begin{equation*}
	\begin{split}
	\abss{\chi_{A_M}\widetilde{w}}_{L^2_{x,\tau}} &\sim\frac{1}{N^{\frac{1}{2}}}\Big(\int\int_{\xi\sim N} \frac{\abs{\xi}}{(\tau+\xi^2)^2}\chi_{A_M}[\widehat{F}_1(y,\tau)]^2 \ d\xi d\tau\Big)^{\frac{1}{2}} \\
	&\lesssim \frac{1}{N^{\frac{1}{2}}M^{\frac{1}{2}}}\abss{F_{1}(y,t)}_{L^2_t},
	\end{split}
	\end{equation*}
	From the definition of $\dot{X}^{0,\frac{1}{2},\infty}$, it follows that
	\[   \Big\| \int_{-\infty}^t e^{i(t-s)\Delta} F_1(s) \ ds \Big\|_{\dot{X}^{0,\frac{1}{2},\infty}} \lesssim \abss{F_1}_{Z_N}. \]
	On the other hand, we consider the space-time Fourier transform 
	\begin{equation*}
	\begin{split}
	\mathcal{F}\int \chi_{(0,\infty)}(t-s) e^{i(t-s)\Delta} F_{2}(s) \ ds &=\frac{\widetilde{F}_{2,M}(\xi,\tau)}{-\tau-\xi^2-i0}.
	\end{split}
	\end{equation*} 
	It follows from the Plancherel's theorem that
	\[\Big\| \int_{-\infty}^t e^{i(t-s)\Delta} F_2(s) \ ds\Big\|_{\dot{X}^{0,\frac{1}{2},\infty}}\lesssim \abss{F_2}_{\dot{X}^{0,-\frac{1}{2},1}},\]
	and the conclusion immediately follows.
\end{proof}
\noindent
We are ready to proof the multilinear estimate. Note that the position of complex conjugates will be significant in the analysis below.
\begin{theorem} \label{thmnon3} For $1\leq i \leq 5$, let $u_i$ represent $u$ or $\bar{u}$. Then we have the following estimates. \\ 
	1). For any $u \in X^{\frac{1}{4}}$,
	\begin{align}
	\Big\| \partial_x\prod_{i=1}^{5}u_i\Big\|_{\dot{Y}^{\frac{1}{4}}} &\lesssim \abss{u}^5_{\dot{X}^{\frac{1}{4}}}, \label{linn4} 
	\end{align}
	2). Let $s\geq \frac{1}{4}$. For any $u \in X^{s}$,
	\begin{align}
	\Big\| \partial_x\prod_{i=1}^{5}u_i\Big\|_{Y^s} &\lesssim \abss{u}^5_{X^s}. \label{linn3}
	\end{align}
\end{theorem}
\begin{proof}
	As before, our goal is to obtain the estimate
	\begin{equation}\label{goal1}
	\sum_N N^{2s+2}\Big\|P_N\prod_{i=1}^{5}u_i\Big\|^2_{Y_N} \lesssim  \abss{u}^2_{\dot{X}^{s}}\abss{u}^8_{\dot{X}^{\frac{1}{4}}}.
	\end{equation}
	First, we split each term in the left-hand side as the sum of all possible frequency interactions:
	\begin{equation*}
	\begin{split}
	N^{2s+2}\Big\|P_N\partial_x\prod_{i=1}^5 u_i\Big\|^2_{Y_N} 
	&\lesssim N^{2s+2}\sum_{N_1,\ldots,N_5}\Big\|P_N\prod_{i=1}^{5}P_{N_i}u_i\Big\|^2_{Y_N}.
	\end{split}
	\end{equation*}
	Assume that $N_1 \geq N_2 \geq \ldots \geq N_5$. Define $c_{N_1,1}=N_1^s\abss{P_{N_1}u}_{X_{N_1}}$ and $c_{N_i,i}=N_i^{\frac{1}{4}}\abss{P_{N_i}u}_{X_{N_i}}$ for $2\leq i \leq 5$.  We make a slight abuse of notation by using $\sum_{N_i}$ for the summation over all possible $N_1,N_2,\ldots ,N_5$ when the restrictions on these numbers are clear. We also will be using the Cauchy-Schwarz inequality \eqref{cs} and Young's inequality \eqref{yi}.\\
	\\
	We split the left-hand side of \eqref{goal1} over four different kinds of frequency interactions:
	\begin{equation*}
	\begin{split}
	&\sum_{N,N_1,\ldots,N_5}N^{s}\Big\|P_N(P_{N_1}\partial_x u_1\prod_{i=2}^{5}P_{N_i}u_i)\Big\|_{Y_N} \\&= \Big(\sum_{I}+\sum_{II}+\sum_{III}+\sum_{IV}\Big)N^{s}\Big\|P_N(P_{N_1}\partial_x u_1\prod_{i=2}^{5}P_{N_i}u_i)\Big\|_{Y_N}.
	\end{split}
	\end{equation*}
	Each of the summations contains certain ranges of $N,N_1,\ldots, N_5$ described by the following cases: \\
	\begin{enumerate}[label=$\Roman*).$,wide, labelwidth=!, labelindent=0pt]
		\item $N\lesssim N_1\sim N_2\sim N_3\sim N_4\sim N_5.$ \\
		By H{\"o}lder  and Cauchy-Schwarz inequalities, we have
		\begin{align*}
		\Big\|P_N\prod_{i=1}^{5}u_i\Big\|_{L^1_xL^2_t} &\lesssim \sum_{N_i}\abss{P_{N_1}u_1}_{L^{\infty}_xL^2_t}\prod_{i=2}^{5}\abss{P_{N_i}u_i}_{L^{4}_xL^{\infty}_t} \\
		&= \sum_{N_i}\frac{1}{N_1^{s+\frac{1}{2}}} \prod_{i=1}^{5}c_{N_i,i} \\
		&\lesssim \Big(\sum_{N_1\gtrsim N}\frac{1}{N_1^{2s+1}}c^2_{N_1,1}\Big)^{\frac{1}{2}}\abss{u}^4_{\dot{X}^{\frac{1}{4}}}.
		\end{align*}
		Summing over $N\in 2^{\mathbb{Z}}$, we see that
		\begin{align*}
		\sum_I N^{2s+1}\Big\|P_N\prod_{i=1}^{5}u_i\Big\|^2_{L^1_xL^2_t} 
		&\lesssim \sum_{N_1}\sum_{N\lesssim N_1}\left(\frac{N}{N_1}\right)^{2s+1}c^2_{N_1,1}\abss{u}^8_{\dot{X}^{\frac{1}{4}}} \\
		&\lesssim \abss{u}^2_{\dot{X}^{s}}\abss{u}^8_{\dot{X}^{\frac{1}{4}}}.
		\end{align*}
		\item $N\sim N_1 \gg N_2\geq N_3\geq N_4\geq N_5.$ \\
		By the bilinear estimate \eqref{bi25} or \eqref{bi3} on $P_{N_1}u_1P_{N_2}u_2$ (depending on the complex conjugates) and Bernstein inequality on $P_{N_5}u_5$, we have that for each fixed $N$,
		\begin{align*}
		\Big\|P_N&\prod_{i=1}^{5}u_i\Big\|_{L^1_xL^2_t} \\ 
		&\lesssim \sum_{N_i}\abss{P_{N_1}u_1P_{N_2}u_2}_{L^2_{x,t}}\prod_{i=3}^{4}\abss{P_{N_i}u_i}_{L^{4}_xL^{\infty}_t}\abss{P_{N_5}u_5}_{L^{\infty}_{x,t}} \\
		&\lesssim \sum_{N_i}\frac{N_5^{\frac{1}{2}}}{N_1^{\frac{1}{2}}}\abss{P_{N_1}u_1}_{X_{N_1}}\abss{P_{N_2}u_2}_{X_{N_2}}\prod_{i=3}^{4}\abss{P_{N_i}u_i}_{L^{4}_xL^{\infty}_t}\abss{P_{N_5}u_5}_{L^{\infty}_{t}L^2_x} \\
		&= \sum_{N_i}\frac{1}{N_1^{s+\frac{1}{2}}}\left(\frac{N_5}{N_2}\right)^{\frac{1}{4}} \prod_{i=1}^{5}c_{N_i,i}. \\
		\intertext{By Young's inequality \eqref{yi}, this term is bounded by} 
		&\lesssim \sum_{N_1\sim N}\frac{1}{N_1^{s+\frac{1}{2}}}c_{N_1,1}\abss{u}^4_{\dot{X}^{\frac{1}{4}}}.
		\end{align*}
		Therefore,
		\begin{equation*}
		\begin{split}
		\sum_{II} N^{2s+1}\Big\|P_N\prod_{i=1}^{5}u_i\Big\|^2_{L^1_xL^2_t} 
		&\lesssim \sum_{N}\Big(\sum_{N_1\sim N}\left(\frac{N}{N_1}\right)^{s+\frac{1}{2}}c_{N_1,1}\Big)^2\abss{u}^8_{\dot{X}^{\frac{1}{4}}} \\
		&\lesssim \Big(\sum_{N_1}\sum_{N\sim N_1}c^2_{N_1,1}\Big)\abss{u}^8_{\dot{X}^{\frac{1}{4}}} \\
		&\sim\abss{u}^2_{\dot{X}^{s}}\abss{u}^8_{\dot{X}^{\frac{1}{4}}}.
		\end{split}
		\end{equation*}
		\item $N\lesssim N_1 \sim N_2\sim N_{j-1} \gg N_j \geq N_5$ where $j=3$ or $j=4$. \\
		This is similar to case $II)$, but instead we use the bilinear estimate on $P_{N_1}u_1P_{N_j}u_j$. 
		\begin{align*}
		\Big\|P_N&\prod_{i=1}^{5}u_i\Big\|_{L^1_xL^2_t}  \\ 	&\lesssim \sum_{N_i}\abss{P_{N_1}u_1P_{N_j}u_j}_{L^2_{x,t}}\prod_{\substack{2\leq i\leq 4 \\ i\not= j}}\abss{P_{N_i}u_i}_{L^{4}_xL^{\infty}_t}\abss{P_{N_5}u_5}_{L^{\infty}_{x,t}} \\
		&\lesssim \sum_{N_i}\frac{N_5^{\frac{1}{2}}}{N_1^{\frac{1}{2}}}\abss{P_{N_1}u_1}_{X_{N_1}}\abss{P_{N_j}u_j}_{X_{N_2}}\prod_{\substack{2\leq i\leq 4 \\ i\not= j}}\abss{P_{N_i}u_i}_{L^{4}_xL^{\infty}_t}\abss{P_{N_5}u_5}_{L^{\infty}_{t}L^2_x} \\
		&\lesssim \sum_{N_i}\frac{1}{N_1^{s+\frac{1}{2}}}\left(\frac{N_5}{N_j}\right)^{\frac{1}{4}}\prod_{i=1}^{5}c_{N_i,i}. 
		\end{align*}
		Applying the Cauchy-Schwarz inequality \eqref{cs} on $\prod_{i=1}^{j-1}c_{N_i,i}$ and Young's inequality \eqref{yi} on $\prod_{i=j}^{5}c_{N_i,i}$, we see that
		\begin{equation*}
		\sum_{N_i}\frac{1}{N_1^{s+\frac{1}{2}}}\left(\frac{N_5}{N_j}\right)^{\frac{1}{4}}\prod_{i=1}^{5}c_{N_i,i} \lesssim \Big(\sum_{N_1\gtrsim N}\frac{1}{N_1^{2s+1}}c_{N_1,1}^2\Big)^{\frac{1}{2}}\abss{u}^8_{\dot{X}^{\frac{1}{4}}}
		\end{equation*}
		Therefore,
		\begin{align*}
		\sum_{III} N^{2s+1}\Big\|P_N\prod_{i=1}^{5}u_i\Big\|^2_{L^1_xL^2_t} 
		&\lesssim \Big(\sum_{N}\sum_{N_1\gtrsim N}\left(\frac{N}{N_1}\right)^{2s+1}c^2_{N_1,1}\Big)\abss{u}^8_{\dot{X}^{\frac{1}{4}}}  \\
		&\sim\abss{u}^2_{\dot{X}^{s}}\abss{u}^8_{\dot{X}^{\frac{1}{4}}}.
		\end{align*}
		\item $N\lesssim N_1\sim N_2\sim N_3\sim N_4 \gg N_5$. \\
		In this case, we will take the number of complex conjugates in $u_1u_2u_3u_4$ into consideration. Note that the positions of conjugates does not matter here. \\
		\begin{enumerate}[label=$\arabic*)$.,wide, labelwidth=!]
			\item $u_1=u_3=u$ and $u_2=u_4=\bar{u}$.
			We divide into further subcases by comparing the sizes between $N$ and $N_5$. \\
			\begin{enumerate}[label=1.\arabic*). ,wide, labelwidth=!]
				\item $N \sim N_5$. \\
                    In this case, we first use H{\"o}lder inequality and then apply the bilinear estimate \eqref{bi3} on $\abss{P_{N_1}u_1P_{N_5}u_5}_{L^2_{x,t}}$
				\begin{align*}
				\Big\|P_N&\prod_{i=1}^{5}u_i\Big\|_{L^1_xL^2_t} \\ 	&\lesssim \sum_{N_i}\abss{P_{N_1}u_1P_{N_5}u_5}_{L^2_{x,t}}\prod_{i=2}^{3}\abss{P_{N_i}u_i}_{L^{4}_xL^{\infty}_t}\abss{P_{N_4}u_4}_{L^{\infty}_{x,t}} \\
				&\lesssim \sum_{N_i}\frac{N_4^{\frac{1}{2}}}{N_1^{\frac{1}{2}}}\abss{P_{N_1}u_1}_{X_{N_1}}\abss{P_{N_5}u_5}_{X_{N_5}}\prod_{i=2}^{3}\abss{P_{N_i}u_i}_{L^{4}_xL^{\infty}_t}\abss{P_{N_4}u_4}_{L^{\infty}_{t}L^2_x} \\
				&\lesssim \sum_{N_i} \frac{N_4^{\frac{1}{4}}}{N_1^{s+\frac{1}{2}}N_5^{\frac{1}{4}}}\prod_{i=1}^5 c_{N_i,i} \\
				&\sim \sum_{N_i} \frac{1}{N_1^{s+\frac{1}{4}}N^{\frac{1}{4}}}\prod_{i=1}^5 c_{N_i,i} \\
				&\lesssim \Big(\sum_{N_1 \gtrsim N}\frac{1}{N_1^{2s+\frac{1}{2}}N^{\frac{1}{2}}} c_{N_1,1}^2\Big)^{\frac{1}{2}}\abss{u}_{\dot{X}^{\frac{1}{4}}}^4,
				\end{align*}
				where we used Cauchy-Schwarz, the fact that $N_4\sim N_1$, $N\sim N_5$ and the trivial inequality $c_{N_5,5} \leq \abss{u}_{\dot{X}^{\frac{1}{4}}}$ in the last step. Consequently, 
				\begin{align*}
				\sum_{\substack{IV \\ N\sim N_5}} N^{2s+1}\Big\|P_N\prod_{i=1}^{5}u_i\Big\|^2_{L^1_xL^2_t} 
				&\lesssim \Big(\sum_{N}\sum_{N_1\gtrsim N}\left(\frac{N}{N_1}\right)^{2s+\frac{1}{2}}c^2_{N_1,1}\Big)\abss{u}^8_{\dot{X}^{\frac{1}{4}}} \\
				&\sim\abss{u}^2_{\dot{X}^{s}}\abss{u}^8_{\dot{X}^{\frac{1}{4}}}.
				\end{align*}
				\item $N\gg N_5$. \\
				We split $\prod_{i=1}^{5}P_{N_i}u_i$ into four terms using low and high frequency projections.
				\begin{equation*}
				\begin{split}
				P_{N_1}u_1P_{N_2}u_2 &= P_{\ll N}(P_{N_1}u_1P_{N_2}u_2)+P_{\gtrsim N} (P_{N_1}u_1P_{N_2}u_2), \\
				P_{N_3}u_3P_{N_4}u_4 &= P_{\ll N}(P_{N_3}u_3P_{N_4}u_4)+P_{\gtrsim N} (P_{N_3}u_3P_{N_4}u_4).
				\end{split}
				\end{equation*}
				Since $N\gg N_5$, so $\prod_{i=1}^{4}P_{N_i}u_i$ must be at frequency $\gg N$. Thus, we can assume that each of the resulting terms after the splits contains at least one high frequency projection. Thus, it suffices to estimate the term: 
                \[P_{\gtrsim N} (P_{N_1}u_1P_{N_2}u_2)\prod_{i=3}^{5}P_{N_i}u_i.\]
                We start by applying the bilinear estimate \eqref{bi25} on $P_{\gtrsim N} (P_{N_1}u_1P_{N_2}u_2)$,
				\begin{equation}\label{nn5}
				\abss{P_{\gtrsim N} (P_{N_1}u_1P_{N_2}u_2)}_{L^2_{x,t}} \lesssim \frac{1}{N^{\frac{1}{2}}}\abss{P_{N_1}u}_{X_{N_1}} \abss{P_{N_2}u}_{X_{N_2}}.
				\end{equation}
				Then, by applying the estimate \eqref{cs} on $c_{N_1,1}c_{N_3,3}c_{N_4,4}$ and \eqref{yi} on $c_{N_2,2}c_{N_5,5}$, we obtain
				\begin{align*}
				\Big\|P_N&\prod_{i=1}^{5}u_i\Big\|_{L^1_xL^2_t} \\
				&\lesssim \sum_{N_i}\abss{P_{\gtrsim N}(P_{N_1}u_1P_{N_2}u_2)}_{L^2_{x,t}}\prod_{i=3}^{4}\abss{P_{N_i}u_i}_{L^{4}_xL^{\infty}_t}\abss{P_{N_5}u_5}_{L^{\infty}_{x,t}} \\
				&\lesssim \sum_{N_i}\frac{N_5^{\frac{1}{2}}}{N^{\frac{1}{2}}}\abss{P_{N_1}u}_{X_{N_1}} \abss{P_{N_2}u}_{X_{N_2}}\prod_{i=3}^{4}\abss{P_{N_i}u_i}_{L^{4}_xL^{\infty}_t}\abss{P_{N_5}u_5}_{L^{\infty}_{t}L^2_x} \stepcounter{equation}\tag{\theequation}\label{nnn5} \\ 
				&\lesssim \sum_{N_i N}\frac{1}{N^{\frac{1}{2}}N_1^{s}} \left(\frac{N_5}{N_2}\right)^{\frac{1}{4}}\prod_{i=1}^{5}c_{N_i,i} \\
				&\sim \sum_{N_i N}\frac{1}{N^{\frac{1}{2}}N_1^{s}} \left(\frac{N_5}{N_3}\right)^{\frac{1}{4}}\prod_{i=1}^{5}c_{N_i,i} \\
				&\lesssim \Big(\sum_{N_1\gtrsim N}\frac{1}{NN_1^{2s}}c^2_{N_1,1}\Big)^{\frac{1}{2}}\abss{u}^4_{\dot{X}^{\frac{1}{4}}}.,
				\end{align*}
				where we used Cauchy-Schwarz on $\sum_{N_i} \frac{1}{N^{\frac{1}{2}}N_1^s}c_{N_1,1}c_{N_2,2}$ and Young's inequality on $\sum_{N_i} \Big(\frac{N_5}{N_3}\Big)^{\frac{1}{4}}c_{N_3,3}c_{N_4,4}c_{N_5,5}$. Therefore,
				\begin{align*}
				\sum_{\substack{IV \\ N\gg N_5}} N^{2s+1}\Big\|P_N\prod_{i=1}^{5}u_i\Big\|^2_{L^1_xL^2_t} 
				&\lesssim \Big(\sum_{N}\sum_{N_1\gtrsim N}\left(\frac{N}{N_1}\right)^{2s}c^2_{N_1,1}\Big)\abss{u}^8_{\dot{X}^{\frac{1}{4}}} \\
				&\sim\abss{u}^2_{\dot{X}^{s}}\abss{u}^8_{\dot{X}^{\frac{1}{4}}}.
				\end{align*}
				\item $N\ll N_5$. \\
				This is similar to case 1.2), but we split $\prod_{i=1}^{5}P_{N_i}u_i$ at $N_5$ instead of $N$.
				\begin{equation*}
				\begin{split}
				P_{N_1}u_1P_{N_2}u_2 &= P_{\ll N_5}(P_{N_1}u_1P_{N_2}u_2)+P_{\gtrsim N_5} (P_{N_1}u_1P_{N_2}u_2), \\
				P_{N_3}u_3P_{N_4}u_4 &= P_{\ll N_5}(P_{N_3}u_3P_{N_4}u_4)+P_{\gtrsim N_5} (P_{N_3}u_3P_{N_4}u_4).
				\end{split}
				\end{equation*}
				Since the output is supported at frequency $N\ll N_5$, we can see that $\prod_{i=1}^{4}P_{N_i}u_i$ must be supported at frequency $\sim N_5$. Thus, we can assume that each term in the product expansion contains at least one high frequency projection. To estimate the product, we can use \eqref{nn5} and \eqref{nnn5} that we just obtained and replace $N^{-\frac{1}{2}}$ by $N_5^{-\frac{1}{2}}$.
				\begin{equation*}
				\begin{split}
				\Big\|P_N\prod_{i=1}^{5}u_i\Big\|_{L^1_xL^2_t} 
				&\lesssim \sum_{N_i}\frac{1}{N_5^{\frac{1}{2}}N_1^{s}} \left(\frac{N_5}{N_3}\right)^{\frac{1}{4}}\prod_{i=1}^{5}c_{N_i,i} \\
				&\ll \sum_{N_i}\frac{1}{N^{\frac{1}{2}}N_1^{s}} \left(\frac{N_5}{N_3}\right)^{\frac{1}{4}}\prod_{i=1}^{5}c_{N_i,i} \\
				&\lesssim \Big(\sum_{N_1\gtrsim N}\frac{1}{NN_1^{2s}}c^2_{N_1,1}\Big)^{\frac{1}{2}}\abss{u}^4_{\dot{X}^{\frac{1}{4}}},
				\end{split}
				\end{equation*}
				which leads to the same result as in the previous case.\\
			\end{enumerate}
			\item $u_1=u_2=u_3=u$, $u_4$ and $u_5$ can be either $u$ or $\bar{u}$. \\
			This is the hardest case and requires some frequency-modulation analysis. Suppose that for some $1\leq j\leq 5$ the space-time Fourier transform of $P_{N_j}u$  is supported in the set
			\begin{subequations}\label{set}
				\begin{align}
				\{(\xi,\tau) &: \abs{\tau+N_1^2}>\frac{1}{32}N_1^2\}, \label{set11}\\
				\intertext{or that of $P_{N_j}\bar{u}$ (for $4\leq j \leq 5$) is supported in the set}
				\{(\xi,\tau) &: \abs{\tau-N_1^2}>\frac{1}{32}N_1^2\}. \label{set12}
				\end{align}
			\end{subequations}
			Then, \eqref{xsb} yields
			\[\abss{P_{N_j}u_j}_{L^2_{x,t}}\lesssim N_1^{-1}\abss{P_{N_j}u_j}_{\dot{X}^{0,\frac{1}{2},\infty}}\lesssim N_1^{-1}\abss{P_{N_j}u_j}_{X_{N_j}}. \]
			Without loss of generality, assume that $j=1$. Then by H{\"o}lder and Bernstein inequalities,
			\begin{align*}
			\Big\|P_N\prod_{i=1}^{5}P_{N_i}u_i\Big\|_{L^1_xL^2_t} 
			&\lesssim \abss{P_{N_1}u_1}_{L^2_{x,t}}\prod_{i=2}^{3}\abss{P_{N_i}u_i}_{L^4_xL^{\infty}_t}\prod_{i=4}^{5}\abss{P_{N_i}u_i}_{L^{\infty}_{x,t}} \\
			&\lesssim \frac{1}{N_1^{s+\frac{1}{2}}}\left(\frac{N_4N_5}{N^2_1}\right)^{\frac{1}{4}}\prod_{i=1}^{5}c_{N_i,i} \\
			&\sim  \frac{1}{N_1^{s+\frac{1}{2}}}\left(\frac{N_5}{N_3}\right)^{\frac{1}{4}}\prod_{i=1}^{5}c_{N_i,i} .
			\end{align*}
			On the other hand, if the space-time Fourier transform of $P_{N_5}u_5$ is supported in the set \eqref{set11} in the case $u_5=u$ or \eqref{set12} in the case $u_5 = \bar{ u}$, then we have
			\begin{align*}
			\Big\|P_N\prod_{i=1}^{5}P_{N_i}u_i\Big\|_{L^1_xL^2_t}
			&\lesssim \prod_{i=1}^2\abss{P_{N_i}u_i}_{L^4_xL^{\infty}_t}\abss{P_{N_3}u_3P_{N_4}u_4P_{N_5}u_5}_{L^{2}_{x,t}} \\ 
			&\lesssim \prod_{i=1}^{2}\abss{P_{N_i}u_i}_{L^4_xL^{\infty}_t}\prod_{i=3}^{4}\abss{P_{N_i}u_i}_{L^{\infty}_{t}L^4_x}\abss{P_{N_5}u_5}_{L^2_{t}L^{\infty}_x} \\
			&\lesssim N_5^{\frac{1}{2}}\prod_{i=1}^{4}\abss{P_{N_i}u_i}_{L^4_xL^{\infty}_t}\abss{P_{N_5}u_5}_{L^2_{x,t}} \\
			&\lesssim \frac{N_5^{\frac{1}{4}}}{N_1^{s+\frac{3}{4}}}\prod_{i=1}^{5}c_{N_i,i} \\
			&\sim \frac{1}{N_1^{s+\frac{1}{2}}}\left(\frac{N_5}{N_3}\right)^{\frac{1}{4}}\prod_{i=1}^{5}c_{N_i,i}.
			\end{align*}
			We then get the desired result by observing that
			\[\frac{1}{N_1^{s+\frac{1}{2}}}\left(\frac{N_5}{N_3}\right)^{\frac{1}{4}}\prod_{i=1}^{5}c_{N_i,i} \lesssim \Big(\sum_{N_1\gtrsim N}\frac{1}{N_1^{2s+1}}c^2_{N_1,1}\Big)^{\frac{1}{2}}\abss{u}^4_{\dot{X}^{\frac{1}{4}}}.\]
			Thus, we can assume that the space-time Fourier transform of $P_{N_j}u$ is supported in the set
			\begin{subequations}\label{set2}
				\begin{align*}
				\{\xi,\tau &: \abs{\tau+N_1^2}\leq \frac{1}{32}N_1^2\}, \\
				\intertext{and that of $P_{N_k}\bar{u}$ is supported in}
				\{\xi,\tau &: \abs{\tau-N_1^2}\leq \frac{1}{32}N_1^2\}.
				\end{align*}
			\end{subequations}
			Here, we introduce Riesz transforms $P_+$ and $P_-$ defined by
			\[\widehat{P_+f}(\xi)=1_{\xi\geq 0}\hat{f}, \ \ \ \ \ \widehat{P_-f}(\xi)=1_{\xi< 0}\hat{f}.\]
			Then, denoting $P_+P_{N_i}:=P^+_{N_i}$ and $P_-P_{N_i}:=P^-_{N_i}$, for $1\leq i \leq 4$, we decompose $P_{N_i}u_i$ into
			\[P_{N_i}u_i=P^+_{N_i}u_i+P^-_{N_i}u_i,\]
			and consider all the terms that we get from $\prod_{i=1}^{5}P_{N_i}u_i$. For any term that contains $P^+_{N_j}uP^-_{N_k}u$, $P^+_{N_j}uP^+_{N_k}\bar{u}$ or $P^-_{N_j}uP^-_{N_k}\bar{u}$, where $1\leq j < k \leq 4$, we can apply the bilinear estimates \eqref{bi25} and \eqref{bi3}, then proceed with the H{\"o}lder's and Bernstein inequality on $L^1_xL^2_t$ as in the previous cases. For example, if  $j=1$ and $k=2$, then we have
			\begin{align*}
			\Big\|P_N(&P^+_{N_1}u_1P^-_{N_2}u_2\prod_{i=3}^{5}P_{N_i}u_i)\Big\|_{L^1_xL^2_t} \\
			&\lesssim \frac{N_5^{\frac{1}{2}}}{N_1^{\frac{1}{2}}}\prod_{i=1}^{2}\abss{P_{N_i}u}_{X_{N_i}}\prod_{i=3}^{4}\abss{P_{N_i}u_i}_{L^4_xL^{\infty}_{t}}\abss{P_{N_5}u_5}_{L^{\infty}_{t}L^2_x} \\
			&\lesssim \frac{1}{N_1^{s+\frac{1}{2}}}\left(\frac{N_5}{N_2}\right)^{\frac{1}{4}}\prod_{i=1}^{5}c_{N_i,i} \\
			&\sim  \frac{1}{N_1^{s+\frac{1}{2}}}\left(\frac{N_5}{N_3}\right)^{\frac{1}{4}}\prod_{i=1}^{5}c_{N_i,i},
			\end{align*}
			Therefore, it suffices to consider the following four terms. 
			\begin{enumerate}
				\item $(\prod_{i=1}^{3}P^+_{N_i}u)P^+_{N_4}uP_{N_5}u_5$
				\item $(\prod_{i=1}^{3}P^-_{N_i}u)P^-_{N_4}uP_{N_5}u_5$
				\item $(\prod_{i=1}^{3}P^+_{N_i}u)P^-_{N_4}\bar{u}P_{N_5}u_5$
				\item $(\prod_{i=1}^{3}P^-_{N_i}u)P^+_{N_4}\bar{u}P_{N_5}u_5$.
			\end{enumerate}
			In either case, simple algebra shows that the space-time Fourier transform of the product is supported at least $\gtrsim N_1^2$ away from the parabola $\tau=-\xi^2$. The worst case is (iii) with $u_5=u$ where the output's modulation is
			\[ (3N_1-N_1\pm N_5)^2-4N_1^2+N_1^2\sim N_1^2.\]
			Thus, we can put these products in the $\dot{X}^{0,-\frac{1}{2},1}$ space and get a good bound. For example, focusing on (iii), we use H{\"o}lder inequality, Bernstein inequality and the boundedness of Riesz transforms.
			\begin{align*}
			\Big\|P_N[(\prod_{i=1}^{3}&P^+_{N_i}u)P^-_{N_4}\bar{u}P_{N_5}u_5]\Big\|_{\dot{X}^{0,-\frac{1}{2},1}}  \\
			&\lesssim \frac{1}{N_1}\Big\|(\prod_{i=1}^{3}P^+_{N_i}u)P^-_{N_4}\bar{u}P_{N_5}u_5\Big\|_{L^2_{t,x}} \\
			&\lesssim \frac{(N_4N_5)^{\frac{1}{2}}}{N_1}\prod_{i=1}^{3}\abss{P_{N_i}u}_{L^6_{t,x}}\prod_{i=4}^{5}\abss{P_{N_i}u}_{L^{\infty}_{t}L^2_{x}} \\
			&\lesssim \frac{1}{N_1^{s+1}}\left(\frac{N_5}{N_1}\right)^{\frac{1}{4}}\prod_{i=1}^{5}c_{N_i,i} \\
			&\sim  \frac{1}{N_1^{s+1}}\left(\frac{N_5}{N_3}\right)^{\frac{1}{4}}\prod_{i=1}^{5}c_{N_i,i}\\
			&\lesssim \Big(\sum_{ N_1\gtrsim N}\frac{1}{N_1^{2s+2}}c^2_{N_1,1}\Big)^{\frac{1}{2}}\abss{u}^4_{\dot{X}^{\frac{1}{4}}}.
			\end{align*}
			Hence, by summing over $N$ and $N_i$'s, we have
			\begin{align*}
			\sum_{IV}N^{2s+2}\Big\|&P_N[(\prod_{i=1}^{3}P^+_{N_i}u)P^-_{N_4}\bar{u}P_{N_5}u_5]\Big\|^2_{\dot{X}^{0,-\frac{1}{2},1}} \\
			&\lesssim \sum_{N_1}\sum_{N\lesssim N_1}\left(\frac{N}{N_1}\right)^{2s+2}c^2_{N_1,1}\abss{u}^8_{\dot{X}^{\frac{1}{4}}} \\
			&\lesssim \abss{u}^2_{\dot{X}^{s}}\abss{u}^8_{\dot{X}^{s_0}},
			\end{align*}
			as desired. \\
			\item $u_1=u_2=u_3=\bar{u}$, $u_4$ and $u_5$ can be either $u$ or $\bar{u}$. \\
			The proof is the same as in the previous case. Note that we get a better result in the sense that the space-time Fourier support of $\prod_{i=1}^{5}P_{N_i}u_i$ when $\mathcal{F}_{x,t}u_i$ is supported in \eqref{set2} is $\gtrsim N_1^2$ away from the parabola $\tau=-\xi^2$ without relying on the Riesz transforms. This concludes the proof of the multilinear estimate.
		\end{enumerate}
	\end{enumerate}
\end{proof}
\section{The Proof of \texorpdfstring{\Cref{gwp1}}{Theorem 1.2}}\label{proofgwp}
\noindent
The proof is similar to what we did in \Cref{sec6} with the same function spaces:
\begin{equation}
\begin{split}\label{norm4}
\abss{u}_{X_N} & =\abss{u}_{L_t^{\infty}L_x^2}+N^{-\frac{1}{4}}\abss{u}_{L_x^{4}L_t^{\infty}}+N^{\frac{1}{2}}\abss{u}_{L_x^{\infty}L_t^2} \\
& \ \ \ +N^{-\frac{1}{2}}\abss{(i\partial_t+\Delta)u}_{L_x^1L_t^2} \\
\abss{u}_{\dot{X}^{s}} & =\Big(\sum_{N\in 2^{\mathbb{Z}}} N^{2s}\abss{P_Nu}^2_{X_N}\Big)^{\frac{1}{2}} \\
\abss{u}_{X^s} &= \abss{u}_{\dot{X}^0}+\abss{u}_{\dot{X}^s} \\
\abss{u}_{Y_N} &= N^{-\frac{1}{2}}\abss{u}_{L_x^1L_t^2} \\
\abss{u}_{\dot{Y}^{s}}&=\Big(\sum_{N\in 2^{\mathbb{Z}}} N^{2s}\abss{P_Nu}^2_{Y_N}\Big)^{\frac{1}{2}} \\
\abss{u}_{Y^s} &= \abss{u}_{\dot{Y}^0}+\abss{u}_{\dot{Y}^s} .
\end{split}
\end{equation}
\noindent
Now we state a multilinear estimate. The proof is shortened as it is similar to that of \Cref{thmnon2} for the most part.
\begin{theorem}  Suppose that $d\geq 5$. Let $s,r>\frac{1}{2}$ and $u_i \in X^{s}$ for $1\leq i \leq d$. Then we have the following estimate: 
	\begin{align}\label{multi}
	\Big\| (\partial_xu_1)\prod_{i=2}^{d}u_i\Big\|_{Y^{r}} &\lesssim \abss{u_1}_{X^{r}}\prod_{i=2}^d\abss{u_i}_{X^{s}}, 
	\end{align}
\end{theorem}
\begin{proof}
	Again, we study the frequency interactions with $N$ being the output frequency and $N_1\geq N_2\geq \ldots \geq N_d$ being the input frequencies. For $s>\frac{1}{2}$, we define $c_{N_1,1}=\abss{P_{N_1} u_1}_{X_{N_1}}$ and $c_{N_i,i}=\abss{P_{N_i}u_i}_{X_{N_i}}$ for $2\leq i \leq d$. We consider the usual $High\times Low \to High$ and $High\times High \to Low$ interactions:\\
	\begin{enumerate}[leftmargin=*]
		\item $N\sim N_1 \gg N_2\geq \ldots \geq N_d$. \\
		With some abuse of notations, we define $\prod_{i=5}^{d-1}A_i=1$ if $d=5$. By H{\"o}lder inequality, Young's inequality and the continuous embedding of function spaces $X^s\hookrightarrow X^{s'}\hookrightarrow \dot{X}^{s'}$ for any $s'>s>\frac{1}{2}$,
		\begin{align*}
            & N^{r-\frac{1}{2}}\Big\| P_N[(P_{N_1}\partial_xu_1)\prod_{i=2}^{d}P_{N_i}u_i]\Big\|_{L^1_xL^2_t} \\
		&\lesssim N^{r-\frac{1}{2}}\sum_{N_i}\abss{P_{N_1}\partial_xu_1}_{L^{\infty}_xL^2_t}\prod_{i=2}^4\abss{P_{N_i}u_i}_{L^{4}_xL^{\infty}_t}\prod_{i=5}^{d-1}\abss{P_{N_i}u_i}_{L^{\infty}_{x,t}}\abss{P_{N_d}u_d}_{L^{4}_xL^{\infty}_t}  \\
		&\lesssim \sum_{N_i}\Big(\frac{N}{N_1}\Big)^{r-\frac{1}{2}}\Big(\frac{N_d}{N_2}\Big)^{\frac{1}{4}}c_{N_1,1}(N_2^{\frac{1}{2}}c_{N_2,2})c_{N_d,d}\prod_{i=3}^{4}N_i^{\frac{1}{4}}c_{N_i,i} \prod_{i=5}^{d-1}N_i^{\frac{1}{2}}c_{N_i,i} \\
		&\lesssim \sum_{N_1 \sim N}\Big(\frac{N}{N_1}\Big)^{r-\frac{1}{2}}c_{N_1,1}\abss{u_2}_{\dot{X}^{\frac{1}{2}}}\prod_{i=3}^{4}\abss{u_i}_{\dot{X}^{\frac{1}{4}}}\prod_{i=5}^{d-1}\abss{u_i}_{\dot{X}^{\frac{1}{2}}}\abss{u_d}_{\dot{X}^{0}} \\
		&\lesssim \sum_{N_1 \sim N}\Big(\frac{N}{N_1}\Big)^{r-\frac{1}{2}}c_{N_1,1}\prod_{i=2}^{d}\abss{u_i}_{X^s}.
		\end{align*}
		Take the $l^2$ summation and \eqref{multi} follows.
		\item $N\lesssim N_1 \sim N_2\geq \ldots \geq N_d$. \\
            This is similar to the previous case, but we apply Cauchy-Schwarz to $\sum_i c_{N_1,1}c_{N_2,2}$ after applying H{\"o}lder inequality.
		\begin{align*}
            &	N^{r-\frac{1}{2}}\Big\| P_N[(P_{N_1}\partial_xu_1)\prod_{i=2}^{d}P_{N_i}u_i]\Big\|_{L^1_xL^2_t} \\
		&\lesssim \sum_{N_i}\Big(\frac{N}{N_1}\Big)^{r-\frac{1}{2}}\Big(\frac{N_d}{N_3}\Big)^{\frac{1}{4}}c_{N_1,1}(N_2^{\frac{1}{4}}c_{N_2,2})(N_3^{\frac{1}{2}}c_{N_3,3})(N_4^{\frac{1}{4}}c_{N_4,4})c_{N_d,d}\prod_{i=5}^{d-1}(N_i^{\frac{1}{2}}c_{N_i,i}) \\
		&\lesssim \Big(\sum_{N_1 \gtrsim N}\Big(\frac{N}{N_1}\Big)^{2r-1}c_{N_1,1}^2\Big)^{\frac{1}{2}}\abss{u_2}_{\dot{X}^{\frac{1}{4}}}\abss{u_3}_{\dot{X}^{\frac{1}{2}}}\abss{u_4}_{\dot{X}^{\frac{1}{4}}}\prod_{i=5}^{d-1}\abss{u_i}_{\dot{X}^{\frac{1}{2}}}\abss{u_d}_{\dot{X}^{0}} \\
		&\lesssim \Big(\sum_{N_1 \gtrsim N}\Big(\frac{N}{N_1}\Big)^{2r-1}\abss{P_{N_1} u_1}^2_{X_{N_1}}\Big)^{\frac{1}{2}}\prod_{i=2}^{d}\abss{u_i}_{X^s}.
		\end{align*}
		Take the $l^2$ summation to obtain \eqref{multi}.
	\end{enumerate}
\end{proof}
\noindent
The proof of \Cref{gwp1} part $(A)$ now follows the same contraction argument as before. To prove part $(B)$ of the theorem,  we replace $u_j$ by $\partial_x u_j$ for some $j\geq 2$, and it follows from \eqref{p1} that $\abss{\partial_x u_i}_{X^s}\lesssim \abss{u_i}_{X^{s+1}}$ for any $s>\frac{1}{2}$. Hence, \eqref{multi} implies that for any $s>\frac{3}{2}$,
\begin{align*}
\Big\| (\partial_x u_1)(\partial_x u_j)\prod_{\substack{i=2 \\ i\not= j}}^{d}u_i\Big\|_{Y^{s}} 
&\lesssim \abss{u_1}_{X^{s}}\abss{\partial_x u_j}_{X^{s-1}}\prod_{\substack{i=2 \\ i\not= j}}^d\abss{u_i}_{X^{s-1}} \\
&\lesssim \prod_{i=1}^d\abss{u_i}_{X^{s}}. 
\end{align*}
Consequently, in the case that a term in $P(u,\bar{u},\partial_x u,\partial_x \bar{u})$ has more than one derivative, we can employ the contraction argument in $X^{s}$. \\
\\
\noindent
\textbf{Acknowledgements. }Part of this work was supported by a National Science Foundation, grant DMS-1600444.\\


\nocite{MR2995102}
\bibliographystyle{elsarticle-harv} 
\bibliography{Pornnopparath110117}




%
\end{document}